\documentclass[12pt]{amsart}
\usepackage{amsmath, amsthm, amssymb}
\usepackage[headings]{fullpage}
\usepackage[pdfstartview={FitH}]{hyperref}

 \numberwithin{equation}{section}
 
 \theoremstyle{plain}
 \newtheorem{theorem}[equation]{Theorem}
 \newtheorem{corollary}[equation]{Corollary}
 \newtheorem{lemma}[equation]{Lemma}
 \newtheorem{proposition}[equation]{Proposition}
 \newtheorem{question}[equation]{Question}
 \theoremstyle{definition}
 \newtheorem{definition}[equation]{Definition}
 
 \newtheorem{example}[equation]{Example}
  
 \newtheorem{remark}[equation]{Remark}

 \DeclareMathOperator{\Hom}{Hom}
 \DeclareMathOperator{\End}{End}
 \DeclareMathOperator{\im}{im}
 
 \DeclareMathOperator{\ann}{ann}
 \DeclareMathOperator{\soc}{soc}
 \DeclareMathOperator{\Spec}{Spec}
 \DeclareMathOperator{\Max}{Max}
 \DeclareMathOperator{\Ext}{Ext}
 \DeclareMathOperator{\Tor}{Tor}
 
 \DeclareMathOperator{\core}{core}
 
 \newcommand{\F}{\mathcal{F}}
 \newcommand{\C}{\mathcal{C}}
 \newcommand{\E}{\mathcal{E}}
 \newcommand{\setS}{\mathcal{S}}
 \newcommand{\G}{\mathcal{G}}
 \newcommand{\M}{\mathfrak{M}}
 \newcommand{\m}{\mathfrak{m}}
 
 \newcommand{\T}{\mathcal{T}}
 \newcommand{\idealizer}{\mathbb{I}}
 
 \newcommand{\separate}{\bigskip}
 
\begin{document}

 \title[A one-sided Prime Ideal Principle for noncommutative rings]
 {A one-sided Prime Ideal Principle\\
 for noncommutative rings}
\author{Manuel L. Reyes}
\address{Department of Mathematics\\
University of California\\
Berkeley, CA 94720, USA}
\email{mreyes@math.berkeley.edu}
\urladdr{\href{http://math.berkeley.edu/~mreyes/}{http://math.berkeley.edu/~mreyes/}}
\thanks{The author was supported in part by a Ford Foundation Predoctoral
Diversity Fellowship. This paper forms part of his Ph.D.\ dissertation at the University
of California, Berkeley.}

\date{December 23, 2010}

\begin{abstract}
\emph{Completely prime right ideals} are introduced as a one-sided generalization
of the concept of a prime ideal in a commutative ring.
Some of their basic properties are investigated, pointing out both similarities
and differences between these right ideals and their commutative counterparts.
We prove the \emph{Completely Prime Ideal Principle}, a theorem stating that
right ideals that are maximal in a specific sense must be completely prime.
We offer a number of applications of the Completely Prime Ideal Principle arising from
many diverse concepts in rings and modules.
These applications show how completely prime right ideals control the one-sided
structure of a ring, and they recover earlier theorems stating that certain
noncommutative rings are domains (namely, proper right PCI rings and rings
with the right restricted minimum condition that are not right artinian).
In order to provide a deeper understanding of the set of completely prime
right ideals in a general ring, we study the special subset of
\emph{comonoform right ideals}.
\end{abstract}

\maketitle

\section{Introduction}
\label{introduction section}

Prime ideals form an important part of the study of commutative algebra. While there
are many reasons why this is so, in this paper we will focus on the fact that
\emph{prime ideals control the structure of commutative rings}. The following two
theorems due to I.\,S.~Cohen (\cite[Thm.~2]{Cohen} and~\cite[Thm.~7]{Cohen}, respectively)
illustrate the kinds of structure theorems that we shall consider.

\begin{theorem}[Cohen's Theorem]
\label{original Cohen's Theorem}
A commutative ring $R$ is noetherian iff every prime ideal of $R$ is finitely
generated.
\end{theorem}

\begin{theorem}[Cohen]
\label{Cohen's invertible theorem}
A commutative ring $R\neq 0$ is a Dedekind domain iff every nonzero prime ideal
of $R$ is invertible.
\end{theorem}

(Notice that Cohen originally stated Theorem~\ref{Cohen's invertible theorem} for
commutative \emph{domains} $R\neq 0$, but it is not hard to extend his result to
arbitrary commutative rings.)
While prime two-sided ideals are studied in noncommutative rings, it is safe to say
that they do not control the structure of noncommutative rings in the sense of the two
theorems above. Part of the trouble is that many complicated rings have few two-sided
ideals. Some of the most dramatic examples are the simple rings, which have only one
prime ideal but often have interesting one-sided structure.

In this paper we propose a remedy to this situation by studying \emph{completely
prime right ideals} (introduced in~\S\ref{completely prime section}), which are
a certain type of ``prime one-sided ideal'' illuminating the structure of a
noncommutative ring.
The idea of prime one-sided ideals is not new, as evidenced by numerous attempts to
define such objects in the literature (for instance, see~\cite{Andrunakievich}, \cite{Koh},
and~\cite{Michler}).
However, a common theme among earlier versions of one-sided prime ideals is that
they were produced by simply deforming the defining condition of a prime ideal in a
commutative ring ($ab\in\mathfrak{p}$ implies $a$ or $b$ lies in $\mathfrak{p}$).
Our approach is slightly less arbitrary, as it is inspired by the systematic analysis
in~\cite{LR} of results from commutative algebra in the vein of
Theorems~\ref{original Cohen's Theorem} and~\ref{Cohen's invertible theorem} above.
As a result, these one-sided primes are accompanied by a ready-made theory producing
a number of results that relate them to the one-sided structure of a ring.

A key part of the classical proofs of Cohen's theorems above is to show that an ideal
of a commutative ring that is maximal with respect to \emph{not} being finitely generated
(respectively, invertible) is prime. 
Many other famous theorems from commutative algebra state that an ideal maximal
with respect to a certain property must be prime. 
The \emph{Prime Ideal Principle} of~\cite{LR} unified a large number of these results,
making use of the fundamental notion of \emph{Oka families of ideals} in commutative
rings.
In~\S\ref{CPIP section}, after introducing Oka families of right ideals, we present
the \emph{Completely Prime Ideal Principle}~\ref{CPIP} (CPIP). This result generalizes
the Prime Ideal Principle of~\cite{LR} to one-sided ideals of a noncommutative ring. 
It formalizes a one-sided ``maximal implies prime'' philosophy: \emph{right ideals that
are maximal in certain senses tend to be completely prime}.
The CPIP is our main tool connecting completely prime right ideals to the (one-sided)
structure of a ring. For instance, it allows us to provide a noncommutative generalization
of Cohen's Theorem~\ref{original Cohen's Theorem} in Theorem~\ref{Cohen's Theorem}.

In order to effectively apply the CPIP, we investigate how to construct examples of
right Oka families (from classes of cyclic modules that are \emph{closed under extensions})
in~\S\ref{cyclic modules section}.
Most of the applications of the Completely Prime Ideal Principle are given
in~\S\ref{applications section}. Some highlights include a study of point annihilators
of modules over noncommutative rings, conditions for a ring to be a domain, a simple
proof that a right PCI ring is a domain, and a one-sided analogue of
Theorem~\ref{Cohen's invertible theorem}.

Finally in~\S\ref{comonoform section} we turn our attention to a special subset of
the completely prime right ideals of a ring, the set of \emph{comonoform right ideals}.
These right ideals are more well-behaved than completely prime right ideals generally
are. They enjoy special versions of the ``Prime Ideal Principle'' and its ``Supplement,''
which again allow us to produce results that relate these right ideals to the one-sided
structure of a noncommutative ring. Their existence is also closely tied to the
well-studied right Gabriel filters from the theory of noncommutative localization.

\subsection*{Conventions}

Throughout this paper, the symbol ``:='' is used to mean that the left-hand side
of the equation is defined to be equal to the right-hand side. All rings are assumed
to be associative with unit element, and all subrings, modules and ring homomorphisms
are assumed to be unital. Fix a ring $R$. We say that $R$ is a \emph{semisimple ring}
if $R_R$ is a semisimple module. We say $R$ is \emph{Dedekind-finite} if every
right invertible element is invertible; this is equivalent to the condition that $R_R$
is not isomorphic to a proper direct summand of itself (see~\cite[Ex.~1.8]{ExercisesModules}).
We write $I_R\subseteq R$ (resp.\ $I\lhd R$) to mean that $I$ is a right
(resp.\ two-sided) ideal in $R$. The term \emph{ideal} always refers to a two-sided
ideal of $R$, with the sole exception of the phrase ``Completely Prime Ideal Principle''
(Theorem~\ref{CPIP}). We let $\Spec(R)$ denote the set of prime (two-sided)  ideals of $R$.
An element of $R$ is \emph{regular} if it is not a left or right zero-divisor.
Given a family $\F$ of right ideals in $R$, we let $\F'$ denote the complement
of $\F$ within the set of all right ideals of $R$, and we let $\Max(\F')$ denote
the set of maximal elements of $\F'$. Given an $R$-module $M_R$, we let $\soc(M)$
denote the socle of $M$ (the sum of all simple submodules of $M$). Finally, we
use ``f.g.''\ as shorthand for ``finitely generated.''

\section{Completely prime right ideals}
\label{completely prime section}

In this section we define the completely prime right ideals that we shall study,
and we investigate some of their basic properties.
Throughout this paper, given an element $m$ and a submodule $N$ of a right
$R$-module $M_R$, we write
\[
m^{-1}N := \{ r \in R : mr \in N \},
\]
which is a right ideal of $R$. Except in~\S\ref{comonoform section},
we only deal with this construction in the form $a^{-1}I$ for an element $a$
and a right ideal $I$ of $R$. In commutative algebra, this ideal is usually
denoted by $(I:a)$.

\begin{definition}
\label{completely prime definition}
A right ideal $P_{R}\subsetneq R$ is \emph{completely prime} if for any
$a,b\in R$ such that $aP\subseteq P$, $ab\in P$ implies that either $a\in P$ or
$b\in P$ (equivalently, for any $a\in R$, $aP\subseteq P$ and $a\notin P$
imply $a^{-1}P=P$).
\end{definition}

Our use of the term ``completely prime'' is justified by the next result,
which characterizes the two-sided ideals that are completely prime as right
ideals. Recall that an ideal $P\lhd R$ is said to be \emph{completely prime}
if the factor ring $R/P$ is a domain (equivalently, $P\neq R$ and for all $a,b\in R$,
$ab\in P\implies a\in P$ or $b\in P$); for instance, see~\cite[p.~194]{FC}. 

\begin{proposition}
\label{two-sided completely primes}
For any ring $R$, an ideal $P\lhd R$ is completely prime as a right ideal
iff it is a completely prime ideal. In particular, an ideal $P$ is completely
prime as a right ideal iff it is completely prime as a left ideal.
\end{proposition}

\begin{proof}
For an ideal $P\lhd R$, we tautologically  have $aP\subseteq P$ for all
$a\in R$. So such $P\neq R$ is completely prime as a right ideal iff for every
$a,b\in R$, $ab\in P$ implies $a\in P$ or $b\in P$, which happens precisely
when $P$ is a completely prime ideal.
\end{proof}

\begin{corollary}
If $R$ is a commutative ring, then an ideal $P\lhd R$ is completely prime as a
(right) ideal iff it is a prime ideal.
\end{corollary}

Thus completely prime right ideals extend the notion of completely prime
ideals in noncommutative rings, and these right ideals also directly generalize
the the concept of a prime ideal of a commutative ring. Some readers may wonder
whether it would be better to reserve the term ``completely prime'' for a right
ideal $P\subsetneq R$ satisfying the following property: for all $a,b\in R$, if
$ab\in P$ then either $a\in P$ or $b\in P$. (Such right ideals have been studied,
for instance, in~\cite{Andrunakievich}.) Let us informally refer to
such right ideals as ``extremely prime.'' We argue that one merit of completely
prime right ideals is that they occur in situations where extremely prime right
ideals are absent. We shall show in Corollary~\ref{maximal right ideal corollary} that every
maximal right ideal of a ring is completely prime, thus proving that every nonzero
ring has a completely prime right ideal.  On the other hand, there are many examples
of nonzero rings that do not have any extremely prime right ideals.
We thank T.\,Y.~Lam for helping to formulate the following result and G.~Bergman
for simplifying its proof.

\begin{proposition}
\label{no extremely primes}
Let $R$ be a simple ring that has a nontrivial idempotent. Then $R$ has no extremely
prime right ideals.
\end{proposition}

\begin{proof}
Assume for contradiction that $I\subsetneq R$ is an extremely prime right ideal.
For any idempotent $e\neq 0, 1$ of $R$ we have $RfR=R$, where $f=1-e\neq 0$.
Since $I$ is extremely prime and $(eRf)^2=0\subseteq I$, we have $eRf\subseteq I$.
Hence $eR=eRfR\subseteq I$, so that $e\in I$. 
Similarly, $f\in I$. Hence $1=e+f\in I$, which is a contradiction.
\end{proof}

Explicit examples of such rings are readily available. Let $k$ be a division ring. Then
we may take $R$ to be the matrix ring $\mathbb{M}_n(k)$ for $n>1$. Alternatively,
if $V_k$ is such that $\dim_k(V)=\alpha$ is any infinite cardinal,
then we may take $R$ to be the factor of $E:=\End_k(V)$ by its unique maximal ideal 
$M=\{g\in E : \dim_k (g(V))<\alpha\}$ (see~\cite[Ex.~3.16]{ExercisesClassical}).
The latter example is certainly neither left nor right noetherian.
More generally, large classes of rings satisfying the hypothesis of
Proposition~\ref{no extremely primes} include simple von Neumann regular rings that
are not division rings, as well as purely infinite simple rings (a ring $R$ that is not a
division ring is \emph{purely infinite simple} if, for every $r\in R$, there exist $x,y\in R$
such that $xry=1$; see~\cite[\S1]{AraGoodearlPardo}).

\separate

It follows from Proposition~\ref{two-sided completely primes} that we can omit
the modifiers ``left'' and ``right'' when referring to two-sided ideals that
are completely prime. The same result shows that the completely prime right ideals
can often be ``sparse'' among two-sided ideals in  noncommutative rings. For
example, there exist many rings that have no completely prime ideals, such as
simple rings that are not domains. Also, it is noteworthy that a prime ideal
$P\lhd R$ of a noncommutative ring is not necessarily a completely prime right
ideal. Thus completely prime right ideals generalize the notion of prime ideals
in commutative rings in a markedly different way than the more familiar two-sided
prime ideals of noncommutative ring theory.
(For further evidence of this idea, see Proposition~\ref{all right ideals completely prime}.)
The point is that these two types of ``primes'' give insight into different facets
of a ring's structure, with completely prime right ideals giving a better picture
of the right-sided structure of a ring as argued throughout this paper.

Below are some alternative characterizations of completely prime right ideals
that help elucidate their nature. The \emph{idealizer} of a right ideal
$J_R\subseteq R$ is the subring of $R$ given by
\[
\idealizer_{R}\left( J\right) :=\left\{ x\in R:xJ\subseteq J\right\}.
\]
This is the largest subring of $R$ in which $J$ is a (two-sided) ideal. It
is a standard fact that $\End_R(R/J)\cong \idealizer_R(J)/J$.

\begin{proposition}
\label{completely prime characterization}
For a right ideal $P_{R}\subsetneq R$,
the following are equivalent:
\begin{itemize}
\item[\normalfont (1)] $P$ is completely prime;
\item[\normalfont (2)] For $a,b\in R$, $ab\in P$ and $a\in \idealizer_R(P)$ imply
either $a\in P$ or $b\in P$;
\item[\normalfont (3)] Any nonzero $f\in \End_R(R/P)$ is injective;
\item[\normalfont (4)] $E:=\End_R(R/P)$ is a domain and ${}_E(R/P)$ is torsionfree.
\end{itemize}
\end{proposition}

\begin{proof}
Characterization (2) is merely a restatement of the definition given above for
completely prime right ideals, so we have (1)$\iff$(2).

(1)$\implies$(3): Let $P$ be a completely prime right ideal, and let
$0\neq f\in\End_R(R/P)$. Choose $x\in R$ such that $f(1+P) =x+P$. Because
$f\neq 0$, $x\notin P$. Also, because $f$ is an $R$-module homomorphism,
$(1+P)P=0$ implies that $(x+P)P=0$, or $xP\subseteq P$. Then because $P$ is
completely prime, $x^{-1}P=P$. But this gives $\ker f=(x^{-1}P)/P=0$, so $f$ is
injective as desired.

(3)$\implies$(1): Assume that any nonzero endomorphism of $R/P$ is injective.
Suppose $x,y\in R$ are such that $xP\subseteq P$ and $xy\in P$. Then there is
an endomorphism $f$ of $R/P$ given by $f(r+P)=xr+P$. If $x\notin P$ then
$f\neq 0$, making $f$ injective. Then $f(y+P)=xy+P=0+P$ implies that $y+P=0+P$,
so that $y\in P$. Hence $P$ is a completely prime right ideal.

(3)$\iff$(4): This equivalence is still true if we replace $R/P$ with any nonzero
module $M_R$. If $E=\End_R(M)$ is a domain and ${}_E M$ is torsionfree, then
it is clear that every nonzero endomorphism of $M$ is injective. Assume
conversely that all nonzero endomorphisms of $M$ are injective. Given
$f,g\in E\setminus\{0\}$ and $m\in M\setminus\{0\}$, injectivity of $g$ gives
$g(m)\neq 0$ and injectivity of $f$ gives $f(g(m))\neq 0$. In particular
$fg\neq 0$, proving that $E$ is a domain. Because $g$ and $m$ above were
arbitrary, we conclude that $M$ is a torsionfree left $E$-module.
\end{proof}

(As a side note, we mention that modules for which every nonzero endomorphism
is injective have been studied by A.\,K.~Tiwary and B.\,M.~Pandeya in~\cite{TiwaryPandeya}.
In~\cite{Xue}, W.~Xue investigated the dual notion of a module for which every
nonzero endomorphism is surjective.  These were respectively referred to as modules with
the properties $(*)$ and $(**)$. In our proof of (3)$\iff$(4) above,
we showed that a module $M_R$ satisfies $(*)$ iff $E:=\End_R(M)$ is a domain
and ${}_E M$ is torsionfree. One can also prove the dual statement that
$M\neq 0$ satisfies $(**)$ iff $E$ is a domain and ${}_E M$ is divisible.)

One consequence of characterization~(3) above is that the property of being
completely prime depends only on the quotient module $R/P$. (Using
language to be introduced in Definition~\ref{similar definition}, the set of
all completely prime right ideals in $R$ is closed under similarity.) It is
a straightforward consequence of~(2) above that {\em a completely prime right
ideal $P$ is a completely prime ideal in the subring $\idealizer_R(P)$}.
This is further evidenced from~(4) because $\idealizer_R(P)/P\cong \End_R(R/P)$ is a domain.

One might wonder whether the torsionfree requirement in condition~(4) above
is necessary. That is, can a cyclic module $C_R$ have endomorphism ring $E$ which
is a domain but with ${}_E C$ not torsionfree? The next example shows that this is indeed
the case.

\begin{example}
For an integer $n>1$, consider the following ring and right ideal:
\[
R := 
\begin{pmatrix}
\mathbb{Z} & \mathbb{Z}/(n)\\
0 & \mathbb{Z}
\end{pmatrix}
\supseteq J_R := 
\begin{pmatrix}
0 & 0\\
0 & \mathbb{Z}
\end{pmatrix}.
\]
It is easy to show that $\idealizer_R(J)=\left(\begin{smallmatrix}\mathbb{Z}&0\\0&\mathbb{Z}\end{smallmatrix}\right)$,
so that $E:=\End_R(R/J)\cong\idealizer_R(J)/J\cong\mathbb{Z}$ is a domain acting
on $R/J\cong\left(\mathbb{Z},\ \mathbb{Z}/(n)\right)_R$
by (left) multiplication. But ${}_E(R/J)$ has nonzero torsion submodule isomorphic to
$\left(0,\ \mathbb{Z}/(n)\right)_R$.
\end{example}

For a prime ideal $P$ in a commutative ring $R$, the factor module $R/P$ is
indecomposable (i.e., has no nontrivial direct summand).  This property persists for completely
prime right ideals.

\begin{corollary}
\label{indecomposable lemma}
If $P$ is a completely prime right ideal of $R$, then the right $R$-module
$R/P$ is indecomposable.
\end{corollary}

\begin{proof}
By Proposition~\ref{completely prime characterization} the ring $E:=\End_R(R/P)$
is a domain. Thus $E$ has no nontrivial idempotents, proving that $R/P$ is
indecomposable.
\end{proof}

For a commutative ring $R$ and $P\in\Spec(R)$, the module $R/P$ is not only
indecomposable, it is uniform. (A module $U_R\neq 0$ is \emph{uniform} if every pair of
nonzero submodules of $U$ has nonzero intersection.) However, the following example
shows that for a completely prime right ideal $P$ in a general ring $R$, $R/P$ need not
be uniform as a right $R$-module.

\begin{example}
\label{completely prime not meet-irreducible}
Let $k$ be a division ring and let $R$ be the following subring of
$\mathbb{M}_3(k)$:
\[
R =
\begin{pmatrix}
k & k & k\\
0 & k & 0\\
0 & 0 & k
\end{pmatrix}
.
\]
Notice that $R$ has precisely three simple right modules (since the
same is true modulo its Jacobson radical), namely $S_i=k$ ($i=1,2,3$)
with right $R$-action given by multiplication by the $(i,i)$-entry of
any matrix in $R$.

Let $P\subseteq R$ be the right ideal consisting of matrices in $R$
whose first row is zero. Notice that $R/P$ is isomorphic to the module
$V:=(k\ k\ k)$ of row vectors over $k$ with the usual right $R$-action.
Then $V$ has unique maximal submodule $M=(0\ k\ k)$, and $M$ in turn
has precisely two nonzero submodules, $U=(0\ k\ 0)$ and $W=(0\ 0\ k)$.
Having cataloged all submodules of $V$, let us list the composition factors:
\[
V/M \cong S_1, \quad M/W \cong U \cong S_2, \quad M/U \cong W \cong S_3.
\]
Notice that $S_1$ occurs as a composition factor of every nonzero factor
module of $V$, and it does not occur as a composition factor of any
proper submodule of $V$. Thus every nonzero endomorphism of $V\cong R/P$
is injective, proving that $P$ is a completely prime right ideal. However,
$V$ contains the direct sum $U\oplus W$, so $R/P\cong V$ is not uniform.
\end{example}

\separate

Proposition~\ref{two-sided completely primes} shows how frequently
completely prime right ideals occur among two-sided ideals. The next few
results give us further insight into how many completely prime right ideals exist
in a general ring. The first result gives a sufficient condition for a right ideal
to be completely prime. Recall that a module is said to be \emph{cohopfian}
if all of its injective endomorphisms are automorphisms. For example, it is
straightforward to show that any artinian module is cohopfian (see~\cite[Ex.~4.16]{ExercisesClassical}).
The following is easily proved using Proposition~\ref{completely prime characterization}(3).

\begin{proposition}
\label{endomorphisms division ring}
If a right ideal $P\subseteq R$ is such that $E:=\End_R(R/P)$ is a division
ring, then $P$ is completely prime. The converse holds if $R/P$ is cohopfian. 
\end{proposition}

\begin{corollary}
\label{maximal right ideal corollary}
\textnormal{(A)} A maximal right ideal $\m_R\subseteq R$ is a completely prime
right ideal.

\noindent\textnormal{(B)} For a right ideal $P$ in a right artinian ring $R$,
the following are equivalent:
\begin{itemize}
\item[\textnormal (1)] $P$ is a completely prime right ideal;
\item[\textnormal (2)] $\End_R(R/P)$ is a division ring;
\item[\textnormal (3)] $P$ is a maximal right (equivalently, maximal left) ideal in
its idealizer $\idealizer_R(P)$.
\end{itemize}
\end{corollary}

\begin{proof}
Part~(A) follows from Schur's Lemma and Proposition~\ref{endomorphisms division ring}.
For part~(B), (1)$\iff$(2) follows from Proposition~\ref{endomorphisms division ring}
(every cyclic right $R$-module is artinian, hence cohopfian), and (2)$\iff$(3) follows easily
from the canonical isomorphism $\idealizer_R(P)/P\cong\End_R(R/P)$.
\end{proof}

Because every nonzero ring has a maximal right ideal (by a familiar Zorn's
lemma argument), part~(A) above applies to show that a nonzero ring always
has a completely prime right ideal. (This fact was already mentioned at the
beginning of this section.)
The same cannot be said for completely prime two-sided ideals or the aforementioned
``extremely prime'' right ideals!
On the other hand, Example~\ref{completely prime not meet-irreducible}
shows that a completely prime right ideal in an artinian ring need not be maximal,
so we cannot hope to strengthen part~(B) very drastically.

To get another indication of the role of completely prime right ideals, we may ask
the following natural question: when is every proper
right ideal of a ring completely prime? It is straightforward to verify that a
commutative ring in which every proper ideal is prime must be a field. 
On the other hand, there exist nonsimple noncommutative rings in which every proper
ideal is prime. (For example, take $R=\End(V_k)$ where $V$ is a right vector
space of dimension at least $\aleph_0$ over a division ring $k$;
see~\cite[Ex. 10.6]{ExercisesClassical}).
In contrast, the behavior of completely prime right ideals is much closer to that
of prime ideals in the commutative case.

\begin{proposition}
\label{all right ideals completely prime}
For a nonzero ring $R$, every proper right ideal in $R$ is completely prime
iff $R$ is a division ring.
\end{proposition}

\begin{proof}
(``Only if'') For an arbitrary $0\neq a\in R$, it suffices to show that
$a$ is right invertible. Assume for contradiction that $aR\neq R$. Then
the right ideal $J=a^2R\subseteq aR\subsetneq R$ is proper and hence is
completely prime. Certainly $a\in \idealizer_R (J)$. Because $a^2\in J$,
we must have $a\in J$ (recall Proposition~\ref{completely prime characterization}(2)).
But also the ideal $0\neq R$ is completely prime, so $R$ is a domain by
Proposition~\ref{two-sided completely primes}. Then $a\in J=a^2R$ implies
$1\in aR$, contradicting that $aR\neq R$.
\end{proof}

An inspection of the proof above actually shows that a nonzero ring $R$
is a division ring iff the endomorphism ring of every nonzero cyclic right
$R$-module is a domain, iff $R$ is a domain and the endomorphism ring of
every cyclic right $R$-module is reduced.
We mention here that certain other ``prime right ideals'' studied previously
do not enjoy the property proved above. For instance, K.~Koh showed~\cite[Thm.~4.2]{Koh}
that all proper right ideals $I$ of a ring $R$ satisfy
\[
(aRb \in I \implies a \in I \text{ or } b \in I) \quad \text{for all } a,b\in R
\]
precisely when $R$ is simple. In this sense, a ring  may have ``too many''
of these prime-like right ideals.

\separate

The following is an analogue of the theorem describing $\Spec(R/I)$ for
a commutative ring $R$ and an ideal $I\lhd R$. We present the result in a
more general context than that of completely prime right ideals because it
will be applicable to other types of ``prime right ideals'' that we will
consider later.

\begin{remark}
\label{spectrum correspondence}

Let $\mathcal{P}$ be a module-theoretic property such that, if $V_R$ is a
module and $I$ is an ideal of $R$ contained in $\ann(V)$, then $V$ satisfies
$\mathcal{P}$ as an $R$-module iff it satisfies $\mathcal{P}$ when
considered as a module over $R/I$. For every ring $R$ let $\setS(R)$ denote
the set of all right ideals $P_R\subseteq R$ such that $R/P$ satisfies
$\mathcal{P}$. Then it follows directly from our assumption on the property
$\mathcal{P}$ that there is a one-to-one correspondence
\[
\{ P_R \in \setS(R) : P \supseteq I \} \longleftrightarrow \setS(R/I)
\]
given by $P\leftrightarrow P/I$.

In particular, we may take $\mathcal{P}$ to be the property ``$V\neq 0$
and every nonzero endomorphism of $V$ is injective.'' Then the
associated set $\setS(R)$ is the collection of all completely prime right
ideals of $R$, according to characterization~(3) of
Proposition~\ref{completely prime characterization}. In this case we conclude
that \emph{for any ideal $I\lhd R$ the completely prime right ideals
of $R/I$ correspond bijectively, in the natural way, to the set of completely
prime right ideals of $R$ containing $I$}.
\end{remark}

In \S\S\ref{CPIP section}--\ref{applications section} we will take a
much closer look at the existence of completely prime right ideals in rings.
To close this section, we explore how completely prime right ideals behave
when ``pulled back'' along ring homomorphisms.
One can interpret Remark~\ref{spectrum correspondence} as demonstrating that,
under a surjective ring homomorphism $f\colon R\to S$, the preimage of any
completely prime right ideal of $S$ is a completely prime right ideal of $R$.
The next example demonstrates that this does not hold for arbitrary ring
homomorphisms.

\begin{example}
\label{completely prime not functorial}
For a division ring $k$, let $S:=\mathbb{M}_3(k)$ and let $R$ be the subring
of $S$ defined in Example~\ref{completely prime not meet-irreducible}.
Consider the right ideals
\[
Q_S := \left\{
\begin{pmatrix}
a & b & c\\
d & e & f\\
d & e & f
\end{pmatrix}
\right\} \subseteq S \quad \text{and} \quad P_R :=
\begin{pmatrix}
k & k & k\\
0 & 0 & 0\\
0 & 0 & 0
\end{pmatrix}
\subseteq R.
\]
Because $Q_S\cong (k\ k\ k)_S^2$ has composition length~2, it is a
maximal right ideal of $S$ and thus is completely prime by
Corollary~\ref{maximal right ideal corollary}(A). Let $f\colon R\to S$
be the inclusion homomorphism. Then $P=Q\cap R=f^{-1}(Q)$. However
$(R/P)_R\cong (0\ k\ k)_R\cong (0\ k\ 0)_R \oplus (0\ 0\ k)_R$ is
decomposable, so Lemma~\ref{indecomposable lemma} shows that $P_R$
is not completely prime.
\end{example}

The above example may seem surprising to the reader who recalls that
completely prime (two-sided) ideals pull back along \emph{any} ring
homomorphism.  The tension between this fact and Example~\ref{completely prime not functorial}
is resolved in the following result.

\begin{proposition}
\label{when completely primes pull back}
Let $f: R\to S$ be a ring homomorphism, let $Q_S\subsetneq S$ be a
completely prime right ideal, and set $P_R:=f^{-1}(Q)$. If
$f(\idealizer_R(P))\subseteq\idealizer_S(Q)$, then $P$ is a completely
prime right ideal of $R$.
\end{proposition}

\begin{proof}
Because $Q$ is a proper right ideal of $S$, $P$ must be a proper right
ideal of $R$. Suppose that $a\in\idealizer_R(P)$ and $b\in R$ are such that
$ab\in P$. Then $f(a)f(b)=f(ab)\in f(P)\subseteq Q$ with
$f(a)\in f(\idealizer_R(P))\subseteq\idealizer_S(Q)$. Because $Q_S$ is
completely prime, this means that one of $f(a)$ or $f(b)$ lies in $Q$. Hence
one of $a$ or $b$ lies in $f^{-1}(Q)=P$ and $P$ is completely prime.
\end{proof}

This simultaneously explains our two positive examples of when a completely
prime right ideal pulls back along a ring homomorphism. When $Q_S\subseteq S$
above is a two-sided ideal, $\idealizer_S(Q)=S$ and the condition
$f(\idealizer_R(P))\subseteq \idealizer_S(Q)$ is trivially satisfied. On
the other hand, if $Q\subseteq \im(f)$ (e.g.\ if $f$ is surjective), then
one can use the fact that $f(P)=Q$ to show that
$f(\idealizer_R(P))\subseteq \idealizer_S(Q)$ again holds.

\section{The Completely Prime Ideal Principle}
\label{CPIP section}

A distinct advantage that completely prime right ideals have over earlier
notions of ``one-sided primes'' is a theorem assuring the existence of
completely prime right ideals in a wide array of situations.
It states that right ideals that are ``maximal'' in certain senses must be
completely prime.
This result is the \emph{Completely Prime Ideal Principle}, or \emph{CPIP},
and it is presented in Theorem~\ref{CPIP} below.

A number of famous theorems from commutative algebra state that an ideal maximal
with respect to a certain property must be prime. A useful perspective from which
to study this phenomenon is to consider a family $\F$ of ideals in a commutative
ring $R$ and ask when an ideal maximal in the \emph{complement} $\F'$ of $\F$ is
prime. Some well-known examples of such $\F$ include the family of ideals intersecting
a fixed multiplicative set $S\subseteq R$, the family of finitely generated
ideals, the family of principal ideals, and the family of ideals that do not
annihilate any nonzero element of a fixed module $M_R$. 
In~\cite{LR} an \emph{Oka family of ideals} in a commutative ring $R$ was defined
to be a set $\F$ of ideals of $R$ with $R\in\F$ such that, for any ideal
$I\lhd R$ and element $a\in R$,  $I+(a)\in\F$ and $(I:a)\in\F$ imply $I\in\F$. The
\emph{Prime Ideal Principle} (or PIP)~\cite[Thm.~2.4]{LR} states that for any
Oka family $\F$, an ideal maximal in the complement of $\F$ is prime (in short,
$\Max(\F')\subseteq\Spec(R)$). In~\cite[\S 3]{LR} it was shown that many of the ``maximal
implies prime'' results in commutative algebra (including those mentioned above) follow
directly from the Prime Ideal Principle. 

The following notion generalizes Oka families to the noncommutative setting.

\begin{definition}
\label{Oka family definition}
Let $R$ be a ring. An \emph{Oka family of right ideals} (or
\emph{right Oka family}) in $R$ is a family $\F$ of right ideals
with $R\in \F$ such that, given any $I_R\subseteq R$
and $a\in R$,
\begin{equation}
\label{Oka property}
I+aR,\ a^{-1}I\in \F\implies I\in \F.
\end{equation}
\end{definition}

If $R$ is commutative, notice that this coincides with the definition of an
Oka family of ideals in $R$, given in~\cite[Def.~2.1]{LR}. When verifying that
some set $\F$ is a right Oka family, we will often omit the step of showing
that $R\in\F$ if this is straightforward.

\begin{remark}
\label{complete lattice}
The fact that this definition is given in terms of the ``closure
property''~\eqref{Oka property} makes it clear that the collection of Oka
families of right ideals in a ring $R$ is closed under arbitrary intersections.
Thus the set of right Oka families of $R$ forms a complete lattice under
the containment relation.
\end{remark}

Without delay, let us prove the noncommutative analogue of the Prime Ideal
Principle~\cite[Thm.~2.4]{LR}, the Completely Prime Ideal Principle (CPIP).

\begin{theorem}[Completely Prime Ideal Principle]
\label{CPIP}
Let $\F$ be an Oka family of right ideals in a ring $R$. Then
every $I\in \Max(\F ')$ is a completely prime right ideal.
\end{theorem}

\begin{proof}
Let $I\in \Max(\F ')$. Notice that $I\neq R$ since $R\in\F$. Assume for
contradiction that there exist $a,b\in R\setminus I$ such that $aI\subseteq I$
and $ab\in I$. Because $a\notin I$ we have $I\subsetneq I+aR$. Additionally
$I\subseteq a^{-1}I$, and $b\in a^{-1}I$ implies $I\subsetneq a^{-1}I$. Since
$I\in\Max(\F')$ we find that $I+aR,\ a^{-1}I \in\F$. Because $\F$ is a right
Oka family we must have $I\in\F$, a contradiction.
\end{proof}

In the original setting of Oka families in commutative rings, a result called
the ``Prime Ideal Principle Supplement''~\cite[Thm.~2.6]{LR} was used to
recover results such as Cohen's Theorem~\cite[Thm.~2]{Cohen} that a commutative
ring $R$ is noetherian iff its prime ideals are all finitely generated.
As with the Completely Prime Ideal Principle above, there is a direct
generalization of this fact for noncommutative rings.
The idea of this result is that for \emph{certain} right Oka families
$\F$, in order to test whether $\F$ contains all right ideals of $R$, it is
sufficient to test only the completely prime right ideals.
We first define the one-sided version of a concept introduced in~\cite{LR}.

\begin{definition}
\label{semifilter definition}
A \emph{semifilter of right ideals} in a ring $R$ is a family $\F$ of right
ideals such, for all right ideals $I$ and $J$ of $R$, if $I\in \F$ and $J\supseteq I$
then $J\in \F$.
\end{definition}

\begin{theorem}[Completely Prime Ideal Principle Supplement]
\label{CPIP supplement}
Let $\F$ be a right Oka family in a ring $R$ such that every nonempty chain
of right ideals in $\F'$ (with respect to inclusion) has an upper bound in $\F'$
(for example, if every right ideal in $\F$ is f.g.). Let $\setS$ denote the
set of completely prime right ideals of $R$.
\begin{itemize}
\item[\normalfont (1)] Let $\F_{0}$ be a semifilter of right ideals in $R$. If
$\setS\cap\F_{0}\subseteq\F$, then $\F_{0}\subseteq \F$.
\item[\normalfont (2)] For $J_{R}\subseteq R$, if all right ideals in $\setS$
containing $J$ (resp.\ properly containing $J$) belong to $\F$, then all right
ideals containing $J$ (resp.\ properly containing $J$) belong to $\F$.
\item[\normalfont (3)] If $\setS\subseteq\F$, then all right ideals of $R$ belong
to $\F$.
\end{itemize}
\end{theorem}

\begin{proof}
For~(1), let $\F_0$ be a semifilter of right ideals and suppose that
$\setS\cap\F_0\subseteq\F$. Assume for contradiction that there exists a right
ideal $I\in\F\setminus\F_0=\F'\cap\F_0$. The hypothesis on $\F'$ allows us
to apply Zorn's lemma to find a right ideal $P\supseteq I$ with $P\in\Max(\F')$,
so $P\in\setS$ by the Completely Prime Ideal Principle~\ref{CPIP}.
Because $\F_0$ is a semifilter containing $I$, we also have $P\in\F_0$. It
follows that $P\in\setS\cap\F_0\setminus\F$, contradicting our hypothesis. 

Parts (2) and (3) follow from~(1) by taking the semifilter $\F_0$ to be,
respectively, the set of all right ideals containing $J$, the set of all
right ideals properly containing $J$, or the set of all right ideals of $R$.
\end{proof}

\separate

We will largely refrain from applying the Completely Prime Ideal Principle
and its Supplement until \S\ref{applications section}, when we will have
enough tools to efficiently construct right Oka families.
However, it seems appropriate to at least give one classical application
to showcase these ideas at work. In~\cite[Appendix]{Nagata}, M.~Nagata gave
a simple proof of Cohen's Theorem~\ref{original Cohen's Theorem} using the
following lemma: if an ideal $I$ and an element
$a$ of a commutative ring $R$ are such that $I+(a)$ and $(I:a)$ are finitely
generated, then $I$ itself is finitely generated. This statement amounts to
saying that the family of finitely generated ideals in a commutative ring is
an Oka family. Nagata cited a paper~\cite{Oka} of K.~Oka as the inspiration
for this result. (In~\cite[p.~3007]{LR} it was pointed out that Oka's
Corollaire~2 is the relevant statement.) This was the reason for the use of
the term \emph{Oka family} in~\cite{LR}. More generally, in~\cite[Prop.~3.16]{LR}
it was shown that, for any infinite cardinal $\alpha $, the family of all
ideals generated by a set of cardinality~$<\alpha$ is Oka. 
The following generalizes this collection of results to the noncommutative
setting. We let $\mu(M)$ denote the smallest cardinal  $\mu$ such that the
module $M_{R}$ can be generated by a set of cardinality $\mu$.

\begin{proposition}
\label{f.g. right ideals}
Let $\alpha$ be an infinite cardinal, and let $\F_{<\alpha}$ be
the set of all right ideals $I_{R}\subseteq R$ with $\mu(I)
<\alpha$. Then $\F_{<\alpha }$ is a right Oka family, and any right ideal
maximal with respect to $\mu(I)\geq\alpha$ is completely prime. In particular,
the set of finitely generated right ideals is a right Oka family; hence a
right ideal maximal with respect to not being finitely generated is prime.
\end{proposition}

\begin{proof}
We first show that $\F_{<\alpha}$ is a right Oka family.
Let $I_R\subseteq R$, $a\in R$ be such that $I+aR,\ a^{-1}I\in \F
_{<\alpha }$. From $\mu(I+aR)<\alpha$ it is straightforward to
verify that there is a right ideal $I_0\subseteq I$ with $\mu(I_0)<\alpha$
such that $I+aR=I_0+aR$. It follows that $I=I_0+a(a^{-1}I)$. Because
$\mu(I_0)<\alpha$ and $\mu(a(a^{-1}I))\leq\mu(a^{-1}I)<\alpha$, we
see that $\mu(I)<\alpha+\alpha=\alpha$. Thus $I\in\F_{<\alpha}$, proving
that $\F_{< \alpha}$ is right Oka. 

If $I$ is a right ideal maximal with respect to $\mu(I)\geq\alpha$ then
$I\in\Max(\F')$ and the CPIP~\ref{CPIP} implies that $I$ is completely
prime. The last sentence follows when we take $\alpha = \aleph_0$.
\end{proof}

This leads to a noncommutative generalization of Cohen's
Theorem~\ref{original Cohen's Theorem} for completely prime right ideals.

\begin{theorem}[A noncommutative Cohen's Theorem]
\label{Cohen's Theorem}
A ring $R$ is right noetherian iff all of its completely prime right ideals are
finitely generated.
\end{theorem}

\begin{proof}
This follows from Proposition~\ref{f.g. right ideals} and the CPIP
Supplement~\ref{CPIP supplement}(3) applied to $\F=\F_{<\aleph_0}$, noticing
that $\F$ tautologically consists of f.g.\ right ideals.
\end{proof}

Notice how quickly the last two results were proved using the CPIP and its Supplement!
This highlights the utility of right Oka families as a framework
from which to study such problems. Of course, other generalizations of
Cohen's Theorem have been proven in the past. 
In a forthcoming paper~\cite{Reyes}, we will apply the methods of right Oka
families developed here to improve upon our generalization of Cohen's
Theorem. We will also develop noncommutative generalizations of the theorems
of Kaplansky which say that a commutative ring is a principal ideal ring iff its
prime ideals are principal, iff it is noetherian and its maximal ideals are
principal. To avoid repetition, we shall wait until~\cite{Reyes} to compare
our version of Cohen's Theorem to earlier generalizations in the literature.

The generalization of Cohen's Theorem~\ref{original Cohen's Theorem} in
Theorem~\ref{Cohen's Theorem} above does not hold if we replace the phrase
``completely prime'' with ``extremely prime''  (as defined in~\S\ref{completely prime section}).
Indeed, using Proposition~\ref{no extremely primes} we showed that there
exist rings $R$ that are not right noetherian with no extremely
prime right ideals. But for such $R$, it is vacuously true that every extremely
prime right ideal of $R$ is finitely generated! This strikingly illustrates the idea that
completely prime right ideals control the right-sided structure of a general ring
better than extremely prime right ideals.

For any cardinal $\beta$, we can also define a family
$\F_{\leq\beta}$ of all right ideals $I$ such that $\mu(I)\leq\beta$.
Letting $\beta^+$ denote the successor cardinal of $\beta$, we see
that $\F_{\leq \beta}=\F_{<\beta^+}$, so we have not sacrificed any
generality in the statement of Proposition~\ref{f.g. right ideals}. In particular,
taking $\beta=\aleph_0$ we see that the family of all countably
generated right ideals is a right Oka family.
The ``maximal implies prime'' result in the case where $R$ is commutative
and $\beta=\aleph_0$ was noted in Exercise~11 of~\cite[p.~8]{KapCommutative}.
The case of larger infinite cardinals $\alpha$ for commutative rings was
proved by Gilmer and Heinzer in~\cite[Prop.~3]{GH}. (In~\cite[p.~3017]{LR},
we mistakenly suggested that this had not been previously observed in the
literature.)

One might wonder whether the obvious analogue of Cohen's Theorem for
right ideals with generating sets of higher cardinalities is also true,
in light of Proposition~\ref{f.g. right ideals}. However, in
the commutative case Gilmer and Heinzer~\cite{GH} have already settled this
in the negative. The rings which serve as their counterexamples are
(commutative) valuation domains.

\separate

It is well-known that Cohen's Theorem~\ref{original Cohen's Theorem} can
be used to prove that if $R$ is a commutative noetherian ring, then the
power series ring $R[[x]]$ is also noetherian.
In~\cite{Michler}, G.~Michler proved a version of Cohen's Theorem
and gave an analogous application of this result to power series over
noncommutative rings.
Here we show that our version of Cohen's Theorem can be applied in the
same way.

\begin{corollary}
If a ring $R$ is right noetherian, then the power series ring $R[[x]]$ is
also right noetherian.
\end{corollary}

\begin{proof}
Let $P$ be a completely prime right ideal of $S:=R[[x]]$;
by Theorem~\ref{Cohen's Theorem} it suffices to show that $P$ is finitely
generated. Let $C_R\subseteq R$ be the right ideal of $R$ consisting of all
constant terms of all power series in $P$. Then $P$ is finitely generated. Choose
power series $f_1,\dots,f_n\in P$ whose constant terms generate $C$,
and set $I:=\sum f_j R\subseteq P$. If $x\in P$ then it is easy to see that $P=I + xS$
is finitely generated. So assume that $x\notin P$. In this case, we claim
that $P=I$. Again we will have $P$ finitely generated and the proof will
be complete. Given $h\in P$, the constant term of $h_0:=h$ is equal to the constant
term of some $g_0=\sum a_{0j} f_j\in I$, where $a_{0j}\in R$. Then $h_0-g_0=xh_1$
for some $h_1\in S$. Notice that $xh_1=h_0-g_0\in P$. Because $P$ is completely
prime with $xP=Px\subseteq P$ and $x\notin P$, it follows that $h_1\in P$. One
can proceed inductively to find $g_i=\sum_{j=1}^n a_{ij}f_j$ ($a_{ij}\in R$) such that
$h_i=g_i+xh_{i+1}$. Hence
$h=\sum_{j=1}^n\left(\sum_{i=0}^{\infty} a_{ij}x^i\right)f_j\in I$. 
\end{proof}

\separate

Before moving on, we mention a sort of ``converse'' to the CPIP~\ref{CPIP}
characterizing exactly which families $\F$ of right ideals are such that $\Max(\F')$
consists of completely prime right ideals. It turns out that a weak form of the
Oka property~\eqref{Oka property} characterizes these families.

\begin{proposition}
Let $\F$ be a family of right ideals in a ring $R$. All right ideals in $\Max(\F')$
are completely prime iff,
for all $I_R\subseteq R$ where every right ideal $J\supsetneq I$
lies in $\F$ and for all elements $a\in\idealizer_R(I)$,
the Oka property~\eqref{Oka property} is satisfied.
\end{proposition}

\begin{proof}
First suppose that $\F$ satisfies property~\eqref{Oka property} for all $I$ and $a$
described above. Then the
proof of the CPIP~\ref{CPIP} applies to show that any right ideal in $\Max(\F')$
is completely prime. Conversely, suppose that $\Max(\F')$ consists of completely
prime right ideals, and let $I_R\subseteq R$ and $a\in\idealizer_R(I)$ be as described
above.
Assume for contradiction that $I\notin\F$. It follows that $I\in\Max(\F')$, so
$I$ is a completely prime right ideal. Because $I\notin\F$ and $I+aR\in\F$, we
see that $a\notin I$. But then the remark at the end of Definition~\ref{completely prime definition} 
shows that $I=a^{-1}I\in\F$, a contradiction. We conclude that in fact $I\in\F$,
completing the proof.
\end{proof}

So for a commutative ring $R$, this result classifies precisely which families
$\F$ of ideals satisfy $\Max(\F')\subseteq\Spec(R)$. (Note: here we can replace
$a\in\idealizer_R(I)$ by $a\in R$.) In fact, the above result was first discovered in
the commutative setting by T.\,Y.~Lam and the present author during the development
of~\cite{LR}, though the result did not appear there.

\section{Right Oka families and classes of cyclic modules}
\label{cyclic modules section}

In order to apply the CPIP~\ref{CPIP}, we need an effective tool for
constructing right Oka families.
The relevant result will be Theorem~\ref{extension correspondence}
below. 
This theorem generalizes one of the most important facts about Oka families
in commutative rings: there is a correspondence between Oka families in a ring $R$
and certain classes of cyclic $R$-modules (to be defined below).
Throughout this paper we use $\M_R$ to denote the class of all right $R$-modules
and $\M_R^c\subseteq \M_R$ to denote the subclass of cyclic $R$-modules.

\begin{definition}
\label{closed under extensions definition}
Let $R$ be any ring. A subclass $\C\subseteq \M_{R}^{c}$ with $0\in \C$
is \emph{closed under extensions} if, for every exact sequence
$0\to L\to M\to N\to 0$ of \underline{cyclic} right $R$-modules,
whenever $L,\, N\in\C$ it follows that $M\in\C$.
\end{definition}

Specifically, it was shown in~\cite[Thm.~4.1]{LR} that for any commutative
ring $R$, the Oka families in $R$ are in bijection with the classes
of cyclic $R$-modules that are closed under extensions. This
correspondence provided many interesting examples of Oka families in
commutative rings. The goal of this section is to show that the Oka families of
right ideals in an arbitrary ring $R$ correspond to the classes of cyclic right
$R$-modules which are closed under extensions.

In a commutative ring $R$, the correspondence described above was given
by associating to any Oka family $\F$ the class
$\C:=\{ M_{R}:M\cong R/I\text{ for some }I\in\F\}$ of cyclic modules.
Then $\F$ is determined by $\C$ because, for an ideal $I$ of $R$, we may recover $I$ from the isomorphism class of
the cyclic module $R/I$ since $I$ is the annihilator of this cyclic
module. (In fact, this works for any family $\F$ of ideals in $R$.) 
However, in a noncommutative ring there can certainly exist right ideals
$I,J\subseteq R$ such that $I\neq J$ but $R/I\cong R/J$ (as right $R$-modules).

\begin{example}
\label{similar maximal example}
Any simple artinian ring $R$ has a single isomorphism class of
simple right modules.  Thus all maximal right ideals $\m_R\subseteq R$ have
isomorphic factor modules $R/\m$. But if $R\cong \mathbb{M}_n(k)$ for a division
ring $k$ and if $n>1$ (i.e.~ $R$ is not a division ring), then there exist
multiple maximal right ideals: we may take $\m_i$ ($i=1,\dots,n$) to correspond
to the right ideal of matrices whose $i$th row is zero. In fact, over an infinite
division ring $k$ even the ring $\mathbb{M}_2(k)$ has infinitely many maximal
right ideals! This is true because, for any $\lambda\in k$, the set of all matrices
of the form
\[
\begin{pmatrix}
a & b\\
\lambda a & \lambda b
\end{pmatrix}
\]
is a maximal right ideal, and these right ideals are distinct for each value
of $\lambda$. (Of course, a similar construction also works over the ring
$\mathbb{M}_n(k)$ for $n>2$.) 
\end{example}

Therefore we do not expect every family $\F$ of right ideals to naturally
correspond to a class of cyclic modules. This prompts the following definition.

\begin{definition}
\label{similar definition}
Two right ideals $I$ and $J$ of a ring $R$ are said to be \emph{similar}
if $R/I\cong R/J$ as right $R$-modules. A family $\F$ of right ideals in
a ring $R$ is \emph{closed under similarity} if, for any similar right ideals
$I_R,J_R\subseteq R$, $I\in \F$ implies $J\in \F$. This is equivalent to
$I\in \F \iff J\in \F$ whenever $R/I\cong R/J$.
\end{definition}

The notion of similarity dates at least as far back as Jacobson's
text~\cite[pp.~33~\&~130]{Jacobson} (although he only studied this
idea in specific classes of rings).
With the appropriate terminology in place, the next fact is easily verified.

\begin{proposition}
\label{similarity bijection}
For any ring $R$, there is a bijective correspondence 
\[
\left\{
\begin{array}{c}
\textnormal{families $\F$ of right}\\
\textnormal{ideals of $R$ that are}\\
\textnormal{closed under similarity}
\end{array}
\right\} \longleftrightarrow \left\{
\begin{array}{c}
\textnormal{classes $\C$ of cyclic right}\\
\textnormal{$R$-modules that are closed}\\
\textnormal{under isomorphism}
\end{array}
\right\}.
\]
For a family $\F$ and a class $\C$ as above, the
correspondence is given by the maps
\begin{align*}
\F &\mapsto \C_{\F}:=\{M_{R}:M\cong R/I \text{ for some }I\in\F\},\\
\C &\mapsto \F_{\C}:=\{I_{R}\subseteq R:R/I\in \C\}.
\end{align*}
\end{proposition}

We will show that every right Oka family is closed under similarity with
the help of the following two-part lemma. The first part describes, up
to isomorphism, the cyclic submodules of a cyclic module $R/I$.
The second part is a rather well-known criterion for two cyclic modules
to be isomorphic.

\begin{lemma}
\label{cyclic module lemmas}
Let $R$ be a ring.
\begin{itemize}
\item[(A)]
For any right ideal $I\subseteq R$ and any element $a\in R$, there is an
isomorphism
\[
R/a^{-1}I \overset{\sim}{\longrightarrow} (I+aR)/I \subseteq R/I
\]
given by $r+a^{-1}I \mapsto ar+I$.
\item[(B)]
Given two right ideals $I_{R},J_{R}\subseteq R$, $R/I\cong R/J$
iff there exists $a\in R$ such that $I+aR=R$ and $a^{-1}I=J$.
\end{itemize}
\end{lemma}

\begin{proof}
Part~(A) is a straightforward application of the First Isomorphism
Theorem.
Proofs for part~(B) can be found, for instance, in~\cite[Prop.~1.3.6]{Cohn}
or~\cite[Ex.~1.30]{ExercisesClassical}. (In fact, it was already observed in
Jacobson's text~\cite[p.~33]{Jacobson}, though in the special setting of PIDs.)
In any case, the ``if'' direction follows from part~(A) above, and
the reader can readily verify the ``only if'' direction.
\end{proof}

These elementary observations are very important for us. The reader should
be aware that we will freely use the isomorphism $R/a^{-1}I \cong (I+aR)/I$
throughout this paper.

\begin{proposition}
\label{similarity-closed characterization}
A family $\F$ of right ideals in a ring $R$ is closed under similarity
iff for any $I_{R}\subseteq R$ and $a\in R$, $I+aR=R$ and $a^{-1}I\in\F$
imply $I\in\F$. In particular, any right Oka family $\F$ is closed under
similarity.
\end{proposition}

\begin{proof}
The first statement follows directly from
Lemma~\ref{cyclic module lemmas}(B), and the second statement follows
from Definition~\ref{Oka family definition} because every right Oka family
contains the unit ideal $R$.
\end{proof}

Thus we see that every right Oka family will indeed correspond, as in
Proposition~\ref{similarity bijection}, to \emph{some} class of cyclic right
modules; it remains to show that they correspond precisely to the classes
that are closed under extensions. We first need to mention one
fact regarding module classes closed under extensions. From the condition
$0\in\C$ and the exact sequence $0\to L\to M\to 0\to 0$ for $L_{R}\cong M_{R}$,
we see that a class $\C$ of cyclic modules closed under extensions is also
closed under isomorphisms. We are now ready to prove the main result of this
section.

\begin{theorem}
\label{extension correspondence}
Given a class $\C$ of cyclic right $R$-modules that is closed under
extensions, the family $\F_{\C}$ is an Oka family of right ideals.
Conversely, given a right Oka family $\F$, the class $\C_{\F}$ of cyclic
right $R$-modules is closed under extensions.
\end{theorem}

\begin{proof}
First suppose that the given class $\C$ is closed under extensions.
Then $R\in \F_{\C}$ because $0\in \C$. So let $I_R\subseteq R$ and
$a\in R$ be such that $I+aR$, $a^{-1}I\in \F_{\C}$. Then $R/(I+aR)$
and $R/a^{-1}I$ lie in $\C$. Moreover, we have an exact sequence
\[
0\rightarrow (I+aR)/I\rightarrow R/I\rightarrow R/(I+aR) \rightarrow 0,
\]
where $(I+aR)/I\cong R/a^{-1}I$ lies in $\C$ (recall that $\C$ is
closed under isomorphisms). Because $\C$ is closed under extensions,
$R/I\in \C$. Thus $I\in \F_{\C}$, proving that $\F_{\C}$ is a right
Oka family.

Now suppose that $\F$ is a right Oka family. That $0\in\C_{\F}$
follows from the fact that $R\in \F$. Consider an exact sequence of
cyclic right $R$-modules 
\[
0\rightarrow L\rightarrow M\rightarrow N\rightarrow 0
\]
where $L,N\in \C_{\F}$, so that there exist $A,B\in\F$ such that
$L\cong R/A$ and $N\cong R/B$. We may identify $M$ up to isomorphism with
$R/I$ for some right ideal $I\subseteq R$. 
Because $L$ is cyclic and embeds in $M\cong R/I$, 
we have $L\cong (I+aR)/I$ for some $a\in R$.
Hence $R/(I+aR)\cong N\cong R/B$, and Proposition~\ref{similarity-closed characterization}
implies that $I+aR\in \F$. Note also that $R/a^{-1}I\cong (I+aR)/I\cong L\cong R/A$,
so by Proposition~\ref{similarity-closed characterization} we conclude that
$a^{-1}I\in \F$. Because $\F$ is a right Oka family we must have $I\in \F$.
So $M\cong R/I$ implies that $M\in \C_{\F}$.
\end{proof}

We examine one consequence of this correspondence. This will require the
following lemma, which compares a class $\C\subseteq\M_R^c$ that is closed
under extensions with its closure under extensions in the larger class $\M_R$.

\begin{lemma}
\label{full closure under extensions}
Let $\C$ be a class of cyclic right $R$-modules that is closed under extensions
(as in Definition~\textnormal{\ref{closed under extensions definition}}), and
let $\overline{\C}$ be its closure under extensions in the class $\M_R$ of all cyclic
right $R$-modules. Then $\C=\overline{\C}\cap\M_R^c$.
\end{lemma}

\begin{proof}
Certainly $\C\subseteq\overline{\C}\cap\M_R^c$. Conversely, suppose that
$M\in\overline{\C}\cap\M_R^c$. Because $M\in\overline{\C}$, there is a filtration
\[
0 = M_0 \subseteq M_1 \subseteq \cdots \subseteq M_n = M
\]
such that each $M_j/M_{j-1}\in\C$. One can then prove by downward
induction that the cyclic modules $M/M_j$ lie in $\C$. So $M\cong M/M_0$
and $M/M_0\in\C$ imply that $M\in\C$.
\end{proof}

\begin{corollary}
Let $\F$ be a right Oka family in a ring $R$. Suppose that
$I_R\subseteq R$ is such that $R/I$ has a filtration
\[
0 = M_0 \subseteq M_1 \subseteq \cdots \subseteq M_n = R/I
\]
where each filtration factor is cyclic and of the form
$M_j/M_{j-1}\cong R/I_j$ for some $I_j\in\F$. Then $I\in\F$.
\end{corollary}

\begin{proof}
Let $\C:=\C_{\F}$, which is closed under extensions by
Theorem~\ref{extension correspondence}. Then the above filtration of
the cyclic module $M = R/I$ has filtration factors isomorphic to the
$R/I_j \in \C$. From Lemma~\ref{full closure under extensions} it
follows that $R/I \in \C$, and thus $I \in \F_{\C} = \F$.
\end{proof}

This implies, for instance, that if a right Oka family $\F$
in a ring $R$ contains all maximal right ideals of $R$, then
it contains all right ideals $I$ such that $R/I$ has finite length.

\separate

We close this section by applying Theorem~\ref{extension correspondence}
to produce a second ``converse'' to the Completely Prime Ideal
Principle~\ref{CPIP}, distinct from the one mentioned at the end
of~\S\ref{CPIP section}.
This result mildly strengthens the CPIP to an ``iff'' statement, saying
that a right ideal $P$ of a ring $R$ is completely prime iff
$P\in\Max(\F')$ for some right Oka family $\F$.
(This was already noted in the commutative case in~\cite[p.~274]{LR2}.)

Let $V_{R}$ be an $R$-module, and define the class 
\begin{equation}
\label{E[V]}
\E[V] := \{ M_{R} : M=0 \text{ or } M \not\hookrightarrow V \}.
\end{equation}
We claim that $\E[V]$ is closed under extensions in $\M_R$. Indeed,
suppose that $0\to L\to M\to N\to 0$ is a short exact sequence in $\M_R$
with $L,N\in\E[V]$. If $L=0$ then $M\cong N$, so that $M\in\E[V]$.
Otherwise $L$ cannot embed in $V$. Because $L\hookrightarrow M$, $M$
cannot embed in $V$, proving $M\in\M_R$. With the class $\E[V]$ in
mind, we prove the second ``converse'' of the CPIP.

\begin{proposition}
\label{CPIP converse}
For any completely prime right ideal $P_R\subseteq R$, there exists an
Oka family $\F$ of right ideals in $R$ such that $P\in \Max(\F')$.
\end{proposition}

\begin{proof}
Let $V_{R}=R/P$, and let $\E[V]$ be as above. Fixing the class
$\C = \E[V]\cap \M_R^c$, set $\F:=\F_{\C}$. By Theorem~\ref{extension correspondence},
$\F$ is a right Oka family. Certainly $P\notin\F$ since $R/P=V\notin\E[V]$,
so it only remains to show the maximality of $P$. Assume for contradiction
that there is a right ideal $I\notin \F$ with $I\supsetneq P$. Then we have
a natural surjection $R/P\twoheadrightarrow R/I$, and because $I\notin\F$
we have $0\neq R/I\hookrightarrow V=R/P$. Composing these maps as 
\[
R/P\twoheadrightarrow R/I\hookrightarrow R/P
\]
gives a nonzero endomorphism $f\in \End(R/P)$ with $\ker f=I/P\neq 0$.
This contradicts characterization~(3) of Proposition~\ref{completely prime characterization},
so we must have $P\in\Max (\F')$ as desired.
\end{proof}

\section{Applications of the Completely Prime Ideal Principle}
\label{applications section}

In this section we will give various applications of the Completely Prime
Ideal Principle. Every application should be viewed as a new source
of completely prime right ideals in a ring or as an application of the notion
of completely prime right ideals (and right Oka families) to study the
one-sided structure of a ring.
The diversity of concepts that interweave with the notion of completely
prime right ideals (via right Oka families) in this section showcases the ubiquity of these objects.
We remind the reader that when verifying that a set
$\F$ of right ideals in $R$ is a right Oka family, we will often skip the
step of checking that $R\in\F$.

\begin{remark}
\label{creating Oka families}
An effective method of creating right Oka families is as follows.
Consider a subclass $\E\subseteq \M_R$ that is closed under extensions
in the full class of right modules $\M_R$. Then $\C=\E\cap \M_R^c$ is a
class of cyclic modules that is closed under extensions. Hence $\F:=\F_{\C}$
is a right Oka family. (Notice that, according to Lemma~\ref{full closure under extensions},
\emph{every} such $\C$ arises this way.)

When working relative to a ring homomorphism, a similar method applies.
Recall that for a ring $k$, a \emph{$k$-ring} $R$ is a ring with a fixed
homomorphism $k \to R$. Given a $k$-ring $R$, let $\E_1$ be any class of
right $k$-modules that is closed under extensions in $\M_k$, and let $\E$
denote the subclass of $\M_R$ consisting of modules that lie in $\E_1$ when
considered as $k$-modules under the map $k \to R$. Then $\E$ is certainly
closed under extensions in $\M_R$, so $\C:=\E\cap \M_R^c$ is closed under
extensions and $\F:=\F_{\C}$ is a right Oka family. 
\end{remark}

\subsection{Point annihilators and zero-divisors}
Point annihilators are basic objects from commutative algebra that
connect the modules over a commutative ring to the ideals of that
ring. Prime ideals play an important role there in the form of
associated primes of a module. Here we study these themes in the
setting of noncommutative rings.

\begin{definition}
For a ring $R$ and a module $M_R\neq 0$, a \emph{point annihilator} of 
$M$ is a right ideal of the form $\ann(m)$ for some $0\neq m\in M$. 
\end{definition}

A standard theorem of commutative algebra states that for a module $M_R$ over
a commutative ring $R$, a maximal point annihilator of $M$ is a prime ideal.
The next result is the direct generalization of this fact. This application
takes advantage of the construction $\E[V]$ presented in~\eqref{E[V]}.

\begin{proposition}
\label{point annihilators} 
Let $R$ be a ring and $M_{R}\neq 0$ an $R$-module. The family $\F$ of right
ideals that are not point annihilators of $M$ is a right Oka family. Thus, a
maximal point annihilator of $M$ is a completely prime right ideal.
\end{proposition}

\begin{proof}
Following the notation of~\eqref{E[V]}, let $\C=\E[M] \cap \M_{R}^{c}$,
which is a class of cyclic modules closed under extensions. Then
$\F_{\C}$ is a right Oka family. But by definition of $\E[M]$, we see
that
\begin{align*}
\F_{\C}^{\prime} &= \{ I_R \subseteq R : 0 \neq R/I \hookrightarrow M \} \\
&= \{ \ann(m) : 0 \neq m \in M \} \\
&= \F'.
\end{align*}
So $\F=\F_{\C}$ is a right Oka family. The last statement follows
from the CPIP~\ref{CPIP}.
\end{proof}

The proof that a maximal point annihilator of a module $M_R$ is completely
prime can also be achieved using the following family:
\[
\F := \{ I_R \subseteq R : \text{for $m\in M$, } mI=0 \implies m=0 \}.
\]
One can show that $\F$ is a right Oka family. Moreover, it is readily
checked that $\Max(\F')$ consists of the maximal point annihilators of $M$.
The CPIP again applies to show that the maximal point annihilators of
$M$ are completely prime. This was essentially the approach taken in the
commutative case in~\cite[Prop.~3.5]{LR}.

As in the theory of modules over commutative rings, one may wish to study
``associated primes'' of a module $M$ over a noncommutative ring $R$. For
a module $M_R$, let us say that a completely prime right ideal $P_R\subseteq R$
is \emph{associated to $M$} if it is a point annihilator of $M$ (equivalently,
if $R/P\hookrightarrow M$). A famous fact from commutative algebra is that
a noetherian module over a commutative ring has only finitely many associated primes;
see~\cite[Thm.~3.1]{Eisenbud}. It is easy to show that the analogous statement
for completely prime right ideals does not hold over noncommutative rings.
For instance, Example~\ref{similar maximal example} provided a ring $R$ with
infinitely many maximal right ideals $\{\m_i\}$ such that the modules $R/\m_i$
were all isomorphic to the same simple module, say $S_R$. Then the $\m_i$ are
infinitely many completely prime right ideals that are associated to the module
$S$ (which is simple and thus noetherian).

In response to this easy example, one may ask whether a noetherian module
has finitely many associated completely prime right ideals \emph{up to similarity}.
Again, the answer is negative. We recall an example used by K.\,R.~Goodearl in~\cite{Goodearl}
to answer a question by Goldie. Let $k$ be a field of characteristic zero and let $D$
be the derivation on the power series ring $k[[y]]$ given by $D=y\frac{d}{dy}$.
Define $R:=k[[y]][x;D]$, a skew polynomial extension. Consider the right module $M_R=R/xR$. Notice that
$M\cong k[[y]]$ as a module over $k[[y]]$. Goodearl showed that the nonzero submodules of $M$ are
precisely the $\bar{y}^i R\cong y^i k[[y]]$ (where $\bar{y}^i=y^i+xR\in M$) and that these submodules are pairwise
nonisomorphic. From the fact that each of these submodules has infinite $k$-dimension
and finite $k$-codimension in $M$, one can easily verify that $M$ (and its nonzero
submodules) are monoform (as in Definition~\ref{(co)monoform definition}). So the
right ideals $\ann(\bar{y}^i)$ are comonoform and thus are completely prime by
Proposition~\ref{comonoform is completely prime} to be proved later. But they are
pairwise nonsimilar because the factor modules $R/\ann(\bar{y}^i)\cong \bar{y}^i R$
are pairwise nonisomorphic.

(In spite of this failure of finiteness, interested readers should note that
O.~Goldman developed a theory of associated primes of noncommutative rings
in which every noetherian module has finitely many associated primes;
see~\cite[Thm.~6.14]{Goldman}. We will not discuss Goldman's prime torsion
theories here but will simply remark that they are related to monoform modules
and comonoform right ideals, which are discussed in the the next section.
See, for instance,~\cite{Storrer}.)

The following is an application of Proposition~\ref{point annihilators}.
For a nonzero module $M_R$, we define the \emph{zero-divisors of $M$} in
$R$ to be the set of all $z\in R$ such that $mz=0$ for some $0\neq m\in M$.
A theorem from commutative algebra states that the set of zero-divisors of
a module over a commutative ring $R$ is equal to the union of
some set of prime ideals. Here we generalize this fact to noetherian right
modules over noncommutative rings.

\begin{corollary}
Let $M_R$ be a module over a ring $R$ such that $R$ satisfies the ACC on point
annihilators of $M$ (e.g., if $M_R$ or $R_R$ is noetherian). Then the set of
zero-divisors of $M$ is a union of completely prime right ideals.
\end{corollary}

\begin{proof}
Let $z\in R$ be a zero-divisor of $M_R$. Then there exists $0\neq m\in M$ such that
$z\in\ann(m)$. Because $R$ satisfies ACC on point annihilators of $M$, there exists
a maximal point annihilator $P_z\subseteq R$ of $M$ containing $\ann(m)$, so that
$z\in\ann(m)\subseteq P_z$. By Proposition~\ref{point annihilators}, $P_z$ is a
completely prime right ideal. Choosing some such $P_z$ for every zero-divisor $z$ on
$M$, we see that the set of zero-divisors of $M$ is equal to $\bigcup_z P_z$.

If $R$ is right noetherian, then the ACC hypothesis is certainly satisfied. Finally,
let us assume that $M_R$ is noetherian and prove that $R$ satisfies ACC on point
annihilators of $M$. Let  $I:=\ann(m_0)\subseteq\ann(m_1)\subseteq\cdots$ be an
ascending chain of point annihilators of $M$ (where $m_i\in M\setminus\{0\}$). Notice
that $R/I\cong m_0R\subseteq M$ is a noetherian module; thus $R$ satisfies ACC on
right ideals containing $I$. It follows that this ascending chain of point annihilators
of $M$ must stabilize.
\end{proof}

Next we shall investigate conditions for a ring to be a domain.
The following fact from commutative algebra was recovered in~\cite[Cor.~3.2]{LR}:
a commutative ring $R$ is a domain iff every nonzero prime ideal of $R$ contains
a regular element. We generalize this result through a natural progression of
ideas, starting with another application of Proposition~\ref{point annihilators}.
Given a ring $R$, we will use the term \emph{right principal annihilator}
to mean a right ideal of the form $I=\ann_r(x)$ for some $x\in R\setminus\{0\}$.
This is just another name for a point annihilator of the module $R_R$, but we
use this term below to evoke the idea of chain conditions on annihilators. Also, by a
\emph{left regular element} of $R$ we mean an element $s\in R$ such that $\ann_{\ell}(s)=0$.

\begin{proposition}
For any nonzero ring $R$, the following are equivalent:
\begin{itemize}
\item[\textnormal (1)] $R$ is a domain;
\item[\textnormal (2)] $R$ satisfies ACC on right principal annihilators, and
for every nonzero completely prime right ideal $P$ of $R$, $P$ is not a right
principal annihilator;
\item[\textnormal (3)] $R$ satisfies ACC on right principal annihilators, and
every nonzero completely prime right ideal of $R$ contains a left regular element.
\end{itemize}
\end{proposition}

\begin{proof}
Certainly (1)$\implies$(3)$\implies$(2), so it suffices to show (2)$\implies$(1).
Let $R$ be as in~(2), and let $\F$ be the family of right ideals of $R$ which
are not point annihilators of the module $R_R$. Then $\F$ is a right Oka family
by Proposition~\ref{point annihilators}. Because every point annihilator of
$R_R$ is a right principal annihilator, the first hypothesis shows that $\F'$
has the ascending chain condition. Furthermore, the second assumption shows
that any nonzero completely prime right ideal of $R$ lies in $\F$. By the CPIP
Supplement~\ref{CPIP supplement}(2), all nonzero right ideals lie in $\F$. It
follows that every nonzero element of $R$ has zero right annihilator, proving
that $R$ is a domain.
\end{proof}

A simple example demonstrates that the chain condition is in fact necessary
for (1)$\iff$(2) above. Indeed, let $k$ be a field and let $R$ be the commutative
$k$-algebra generated by $\{x_i : i\in\mathbb{N}\}$ with relations $x_i^2=0$.
Clearly $R$ is not a domain, but its unique prime ideal $(x_0, x_1, x_2, \dots )$
is not a principal annihilator.

This leaves us with the following question: If every completely prime right
ideal of a ring contains a left regular element, then is $R$ a domain?
Professor G.~Bergman has answered this question in the affirmative. With his
kind permission, we present a modified version of his argument below.

\begin{lemma}
\label{non local point annihilators}
For a ring $R$ and a module $M_R$, let $\F$ be the family of right ideals
$I$ of $R$ such that there exists a nonempty finite subset $X\subseteq I$
such that, for all $m\in M$, $mX=0 \implies m=0$. The family $\F$ is a
right Oka family.
\end{lemma}

\begin{proof}
To see that $R\in\F$, simply take $X=\{1\}\subseteq R$. Now suppose that
$I_R\subseteq R$ and $a\in R$ are such that $I+aR,\ a^{-1}I\in\F$. Choose
nonempty subsets $X_0=\{i_1+ar_1,\dots,i_p+ar_p\}\subseteq I+aR$ (where each
$i_k\in I$) and $X_1=\{x_1,\dots, x_q\}\subseteq a^{-1}I$ such that, for
$m\in M$, $mX_j=0$ implies $m=0$ (for $j=0,1$). Define
\[
X := \{ i_1, \dots, i_p, ax_1, \dots, ax_q \} \subseteq I.
\]
Suppose that $mX=0$ for some $m\in M$. Then $maX_1\subseteq mX=0$ implies
that $ma=0$. It follows that $mX_0=0$, from which we conclude $m=0$. This
proves that $I\in\F$, hence $\F$ is right Oka.
\end{proof}

\begin{proposition}
\label{torsionfree characterization}
For a module $M_R\neq 0$ over a ring $R$, the following are equivalent:
\begin{itemize}
\item[\textnormal (1)] $M$ has no zero-divisors (i.e., $0\neq m\in M$ and
$0\neq r\in R$ imply $mr\neq 0$);
\item[\textnormal (2)] Every nonzero completely prime right ideal of $R$
contains a non zero-divisor for $M$;
\item[\textnormal (3)] Every nonzero completely prime right ideal $P$ of $R$
has a nonempty finite subset $X\subseteq P$ such that, for all $m\in M$,
$mX=0 \implies m=0$.
\end{itemize}
\end{proposition}

\begin{proof}
Clearly (1)$\implies$(2)$\implies$(3); we prove (3)$\implies$(1). Assume
that~(3) holds, and let $\F$ be the Oka family of right ideals defined
in Lemma~\ref{non local point annihilators}. It is easy to check that the
union of any chain of right ideals in $\F'$ also lies in $\F'$. By~(3),
every nonzero completely prime right ideal of $R$ lies in $\F$. Then the
CPIP Supplement~\ref{CPIP supplement} implies that all nonzero right ideals
of $R$ lie in $\F$. It is clear that no right ideal in $\F$ can be a point
annihilator for $M$. It follows immediately that $M$ has no zero-divisors.
\end{proof}

\begin{corollary}
\label{domain characterization}
For a ring $R\neq 0$, the following are equivalent:
\begin{itemize}
\item[\textnormal (1)] $R$ is a domain;
\item[\textnormal (2)] Every nonzero completely prime right ideal of $R$
contains a left regular element;
\item[\textnormal (3)] Every nonzero completely prime right ideal of $R$ has
a nonempty finite subset whose left annihilator is zero.
\end{itemize}
\end{corollary}

Here is another demonstration that completely prime right ideals control the
structure of a ring better than the ``extremely prime'' right ideals
(discussed in~\S\ref{completely prime section}). Using
Proposition~\ref{no extremely primes} we constructed rings with no
extremely prime right ideals that are not domains. But it is vacuously true
that every extremely prime right ideal of such a ring contains a regular element.
Thus there is no hope that the result above could be achieved using this more
sparse collection of one-sided primes.

\subsection{Homological properties}
Module-theoretic properties that are preserved under extensions arise very
naturally in homological algebra. This provides a rich supply of right Oka families,
and consequently produces completely prime right ideals via the CPIP.

\begin{example}
\label{proj/inj/flat example}
For a ring $k$ and a $k$-ring $R$, consider the following properties of a right ideal
$I_R\subseteq R$ (which are known to be preserved by extensions of the factor module):
\begin{enumerate}
\item $R/I$ is a projective right $k$-module;
\item $R/I$ is an injective right $k$-module;
\item $R/I$ is a flat right $k$-module.
\end{enumerate}
For each property above, the family $\F$ of all right ideals with that property
is a right Oka family (by Remark~\ref{creating Oka families}); hence $\Max(\F')$
consists of completely prime right ideals.
\end{example}

We have the following immediate application, which includes a criterion
for a ring to be semisimple.

\begin{proposition}
\label{family of direct summands}
The family $\F$ of right ideals that are direct summands of $R_R$ is a
right Oka family. A right ideal $I_{R}\subseteq R$ maximal with respect
to not being a direct summand of $R$ is a maximal right ideal.
A ring $R$ is semisimple iff every maximal right ideal of $R$ is a
direct summand of $R_R$.
\end{proposition}

\begin{proof}
This family $\F$ is readily seen to be equal to the family given in
Example~\ref{proj/inj/flat example} (with $k=R$ and the identity map
$k\to R$), and thus it is right Oka.
Let $P\in \Max(\F')$. Then $P$ is completely prime, so $R/P$ is
indecomposable by Corollary~\ref{indecomposable lemma}. On the other
hand, because every right ideal properly containing $P$ splits in $R_R$, the
module $R/P$ is semisimple. It follows that $R/P$ is simple, so $P$ is maximal
as claimed.

The nontrivial part of the last sentence is the ``if'' direction.
Assume that every maximal right ideal of $R$ is a
direct summand. It suffices to show that every completely prime right
ideal of $R$ is maximal. (For if this is the case, then every completely
prime right ideal will lie the right Oka family $\F$.
Now $\F$ consists of principal---hence f.g.---right ideals by the classical
fact that $\F=\{eR:e^2=e\in R\}$. Then the CPIP Supplement~\ref{CPIP supplement}(3)
will show that every right ideal of $R$ is a direct summand, making $R$
semisimple.)
So suppose $P_R\subsetneq R$ is completely prime. Fix a maximal right ideal
$\m$ of $R$ with $\m\supseteq P$. Because $\m$ is a proper direct summand
of $R_R$, $\m/P$ must be a proper summand of $R/P$. But $R/P$ is indecomposable
by Proposition~\ref{indecomposable lemma}. Thus $\m/P=0$, so that $P=\m$ is maximal.
\end{proof}

Of course, we can also prove the ``iff'' statement above without any reference to right
Oka families. (Suppose that every maximal right ideal of $R$ is a direct
summand. Assume for contradiction that the right socle $S_{R}:=\soc(R_R)$
is a proper right ideal. Then there is some maximal right ideal $\m\subseteq R$
such that $S\subseteq\m$. But by hypothesis there exists $V_R\subseteq R$
such that $R=V\oplus \m$. Then $V\cong R/\m$ is simple. So $V\subseteq S$,
contradicting the fact that $V\cap S\subseteq V\cap \m=0$.) Although such
ad hoc methods are able to recover this fact, our method involving the
CPIP~\ref{CPIP} has the desirable effect of fitting the result into a larger
context. Also, the CPIP and right Oka families may point one to results that
might not have otherwise been discovered without this viewpoint, even if
these results could have been proven individually with other methods.

We can also use Example~\ref{proj/inj/flat example} to recover a bit of the
structure theory of right PCI rings. A right module over a ring $R$ is called
a \emph{proper cyclic module} if it is cyclic and not isomorphic to
$R_R$. (Note that this is stronger than saying that the module is isomorphic
to $R/I$ for some $0\neq I_R\subseteq R$, though it is easy to
confuse the two notions.) A ring $R$ is a \emph{right PCI ring} if every
proper cyclic right $R$-module is injective, and such a ring $R$ is called
a \emph{proper right PCI ring} if it is not right self-injective (by a
theorem of Osofsky, this is equivalent to saying that $R$ is not semisimple).
C.~Faith showed in~\cite{Faith} that any proper right PCI ring is a simple right
semihereditary right Ore domain. In Faith's own words~\cite[p.~98]{Faith},
``The reductions to the case $R$ is a domain are long, and not entirely
satisfactory inasmuch as they are quite intricate.''  Our next application of
the Completely Prime Ideal Principle shows how to easily deduce that a proper
right PCI ring is a domain with the help of a later result on right PCI rings.

\begin{proposition}[Faith]
\label{PCI domain}
A proper right PCI ring is a domain.
\end{proposition}

\begin{proof}
A theorem of R.\,F.~Damiano~\cite{Damiano} states that any right PCI
ring is right noetherian. (Another proof of this result, due to
B.\,L.~Osofsky and P.\,F.~Smith, appears in~\cite[Cor.~7]{OsofskySmith}.)
In particular, any right PCI ring is Dedekind-finite.

Now let $R$ be a proper right PCI ring. Because $R$ is Dedekind-finite,
for every nonzero right ideal $I$, $R/I\ncong R_R$ is a proper cyclic
module. Letting $\F$ denote the family of right ideals $I$ such that $R/I$
is injective, we have $0\in\Max(\F')$. But $\F$ is a right Oka family
by Example~\ref{proj/inj/flat example}, so the CPIP~\ref{CPIP} and
Proposition~\ref{two-sided completely primes} together show that $R$ is a domain.
\end{proof}

The astute reader may worry that the above proof is nothing more than
circular reasoning, because the proof of Damiano's theorem in~\cite{Damiano}
seems to rely on Faith's result! This would indeed be the case if Damiano's
were the only proof available for his theorem. (Specifically, Damiano cites
another result of Faith---basically~\cite[Prop.~16A]{Faith}---to conclude that
over a right PCI ring, every finitely presented proper cyclic module  has a
von Neumann regular endomorphism ring. But Faith's result is stated only for
cyclic \emph{singular} finitely presented modules. So Damiano seems to be
implicitly applying the fact that a proper right PCI ring is a right Ore
domain.) Thankfully, we are saved by the fact that Osofsky and Smith's
(considerably shorter) proof~\cite[Cor. 7]{OsofskySmith} of Damiano's result
does not require any of Faith's structure theory. 

It is worth noting that A.\,K.~Boyle had already provided a proof~\cite[Cor.~9]{Boyle}
that a right noetherian proper right PCI ring is a domain. (This was before
Damiano's theorem had been proved.) One difference between our approach and that
of~\cite{Boyle} is that we do not use any facts about direct sum decompositions of injective
modules over right noetherian rings. Of course, the proof using the CPIP is
also desirable because we are able to fit the result into a larger context in
which it becomes ``natural'' that such a ring should be a domain.

\separate

As in~\cite{LR}, we can generalize Example~\ref{proj/inj/flat example}
with items (1)--(3) below. One may think of the following examples as being
defined by the existence of certain (co)resolutions of the modules. Recall that
a module $M_R$ is said to be \emph{finitely presented} if there exists an exact
sequence of the form $R^m\to R^n\to M\to 0$, and that a \emph{finite free resolution}
of $M$ is an exact sequence of the form
\[
0 \to F_n \to \cdots \to F_1 \to F_0 \to M \to 0
\]
where the $F_i$ are finitely generated free modules.

\begin{example}
\label{proj/inj/flat dim example}
Let $R$ be a $k$-ring, and fix any one of the following properties of a right ideal $I$
in $R$ (known to be closed under extensions), where $n$ is a nonnegative integer:
\begin{itemize}
\item[(1)] $R/I$ has $k$-projective dimension~$\leq n$ (or~$<\infty$);
\item[(2)] $R/I$ has $k$-injective dimension~$\leq n$ (or~$<\infty$);
\item[(3)] $R/I$ has $k$-flat dimension~$\leq n$ (or~$<\infty$);
\item[(4)] $R/I$ is a finitely presented right $R$-modules;
\item[(5)] $R/I$ has a finite free resolution as a right $R$-module.
\end{itemize}
Then the family $\F$ of right ideals satisfying that property is right Oka (as in
Remark~\ref{creating Oka families}), and $\Max(\F')$ consists of completely prime
right ideals.
\end{example}

The families in Example~\ref{proj/inj/flat example} are just (1)--(3) above
with $n=0$. Restricting to the case $n=1$ and $k=R$, the family obtained
from part (1) (resp.\ part (3)) of Example~\ref{proj/inj/flat dim example} 
is the family of projective (resp.\ flat) right ideals of $R$ (for the latter,
see~\cite[(4.86)(2)]{Lectures}). In particular, the CPIP~\ref{CPIP} implies
that \emph{a right ideal of $R$ maximal with respect to not being projective (resp.\ flat)
over $R$ is completely prime.}
 
The family $\F$ in part~(4) above is actually equal to the family of
\emph{finitely generated} right ideals.
Indeed, if $I_R\subseteq R_R$ is finitely generated, then the module $R/I$ is
certainly finitely presented. Conversely, if $R/I$ is a finitely presented
module,~\cite[(4.26)(b)]{Lectures} implies that $I_R$ is f.g. This recovers
the last sentence of Proposition~\ref{f.g. right ideals} from a module-theoretic
perspective.

We can also use the family in~(2) above to generalize Proposition~\ref{PCI domain}
about proper right PCI rings. Given a nonnegative integer $n$, let us say that a ring
$R$ is a \emph{right $n$-PCI ring} if the supremum of the injective dimensions of all
proper cyclic right $R$-modules is equal to~$n$.
(Thus a right $0$-PCI ring is simply a right PCI ring.) Also, call a
right $n$-PCI ring $R$ \emph{proper} if $R_R$ has injective dimension greater than $n$
(possibly infinite). Then the following is proved as in~\ref{PCI domain}, using the family
from Example~\ref{proj/inj/flat dim example}(2).

\begin{proposition}
If a proper right $n$-PCI ring is Dedekind-finite, then it is a domain.
\end{proposition}

Unlike the above result, Proposition~\ref{PCI domain} is not a conditional statement
because Damiano's theorem guarantees that a right $0$-PCI ring is right noetherian,
hence Dedekind-finite. Also, it is known that right PCI rings are right hereditary,
which implies that the injective dimension of $R_R$ is~$1$ if $R$ is proper right $0$-PCI.
Thus we pose the following questions.

\begin{question}
What aspects of the Faith-Damiano structure theory for PCI rings carry over to $n$-PCI
rings?  In particular, for a proper right $n$-PCI ring $R$, we ask:
\begin{enumerate}
\item Must $R_R$ have finite injective dimension? If so, is this dimension necessarily
equal to $n+1$?
\item Must $R$ be Dedekind-finite, or even possibly right noetherian?
What if we assume that $R_R$ has finite injective dimension, say equal to $n+1$
(if~\textnormal{(1)} above fails)?
\end{enumerate}
\end{question}

\separate

Further generalizing Examples~\ref{proj/inj/flat example} and~\ref{proj/inj/flat dim example},
we have the following.

\begin{example}
\label{Ext/Tor example}
Given a $k$-ring $R$, fix a right module $M_k$ and a left module ${}_kN$,
and let $n$ be a nonnegative integer. Fix one of the following properties of
a right ideal $I\subseteq R$:
\begin{itemize}
\item[(1)] $R/I$ satisfies $\Ext_k^n(R/I,M)=0$;
\item[(2)] $R/I$ satisfies $\Ext_k^n(M,R/I)=0$;
\item[(3)] $R/I$ satisfies $\Tor_n^k(R/I,N)=0$.
\end{itemize}
Applying Remark~\ref{creating Oka families}, the family $\F$ of right ideals
satisfying that fixed property is right Oka. Thus $\Max(\F')$ consists of completely
prime right ideals.
\end{example}

(The fact that the corresponding classes of cyclic modules are closed under extensions
follows from a simple analysis of the long exact sequences for $\Ext$ and $\Tor$ 
derived from a short exact sequence $0\rightarrow A\rightarrow
B\rightarrow C\rightarrow 0$ in $\M_R$.)

We can actually use these to recover the families (1)--(3) of
Example~\ref{proj/inj/flat dim example} as follows. It is known that a
module $B$ has $k$-projective dimension $n$ iff $\Ext_k^n(B,M)=0$ for all right
modules $M_k$. Then intersecting the families in Example~\ref{Ext/Tor example}(1)
over all modules $M_k$ gives the class in Example~\ref{proj/inj/flat dim example}(1).
A similar process works for~(2) and~(3) of Example~\ref{proj/inj/flat dim example}.

As an application of case (1) above, we present the following interesting
family of right ideals in any ring associated to an arbitrary module $M_R$.

\begin{proposition}
\label{ideals extending morphisms}
For any module $M_R$, the family $\F$ of all right ideals $I_R\subseteq R$
such that any homomorphism $f\colon I\rightarrow M$ extends to some
$\tilde{f}\colon R\rightarrow M$ is a right Oka family.
A right ideal maximal with respect to $I\notin\F$ is completely prime.
\end{proposition}

\begin{proof}
Let $\G$ be the family in Example~\ref{Ext/Tor example}(1) with $k=R$ and $n=1$.
We claim that $\F=\G$, from which the proposition will certainly follow.
Given $I_R\subseteq R$, consider the long exact sequence in $\Ext$ associated to
the short exact sequence $0\rightarrow I\rightarrow R\rightarrow R/I\rightarrow 0$:
\begin{align*}
0\rightarrow & \Hom_R(R/I,M)\rightarrow \Hom_R(R,M)\rightarrow \Hom_R(I,M) \\
\rightarrow & \Ext_R^1(R/I,M)\rightarrow\Ext_R^1(R,M)=0
\end{align*}
($\Ext_R^1(R,M)=0$ because $R_R$ is projective). Thus $I\in\F$ iff
the natural map $\Hom_R(R,M)\rightarrow \Hom_R(I,M)$ is surjective, iff its
cokernel $\Ext_R^1(R/I,M)$ is zero, iff $I\in\G$.
\end{proof}

It is an interesting exercise to ``check by hand'' that the family $\F$ above
satisfies the Oka property~\eqref{Oka property}. When $R$ is a right self-injective
ring, one can dualize the above proof of Proposition~\ref{ideals extending morphisms},
($R_R$ must be injective to ensure that $\Ext_R^1(M,R)=0$), and a similar argument
works using the functor $\Tor_1$ in place of $\Ext^1$. We obtain the following.

\begin{proposition}
\label{tor 1 family}
\textnormal{(A)} Let $R$ be a right self-injective ring and let $M_R$ be any module. The
family $\F$ of right ideals $I\subseteq R$ such that every homomorphism
$f:M\to R/I$ lifts to some $f':M\to R$ is a right Oka family. Hence, any $I\in\Max(\F')$
is completely prime.

\textnormal{(B)} For a $R$ and a module ${}_RN$, let $\F$ be the family of right ideals
$I_R\subseteq R$ such that the natural map $I\otimes_R N\to R\otimes_R N\cong N$
is injective. Then $\F$ is an Oka family of right ideals. Hence, any $I\in\Max(\F')$ is
completely prime.
\end{proposition}

For us, what is most interesting about Propositions~\ref{ideals extending morphisms}
and~\ref{tor 1 family} is that they provide multiple ways to define right Oka families
starting with \emph{any given module $M_R$.} Thanks to the Completely Prime Ideal
Principle~\ref{CPIP}, each of these families $\F$ gives rise to completely prime right ideals
in $\Max(\F')$ whenever this set is nonempty. 

\subsection{Finiteness conditions, multiplicative sets, and invertibility}
The final few applications of the Completely Prime Ideal Principle given here
come from finiteness conditions on modules, multiplicatively closed subsets of
a ring, and invertible right ideals.

We first turn our attention to finiteness conditions. We remind the reader that
a module $M_R$ is said to be \emph{finitely cogenerated} provided that, for every
set $\{N_i:i\in I\}$ of submodules of $M$ such that $\bigcap_{i\in I} N_i = 0$,
there exists a finite subset $J\subseteq I$ such that $\bigcap_{j\in J} N_j =0$.
This is equivalent to saying that the socle of $M$ is finitely generated and
is an essential submodule of $M$. See~\cite[\S 19A]{Lectures} for further details.

\begin{example}
\label{finiteness conditions}
Let $R$ be a $k$-ring. Fix any one of the following properties of a right ideal
$I$ of $R$:
\begin{itemize}
\item[(1A)] $R/I$ is a finitely generated right $k$-module;
\item[(1B)] $R/I$ is a finitely cogenerated right $k$-module;
\item[(2)] $R/I$ has cardinality~$<\alpha $ for some infinite
cardinal $\alpha$;
\item[(3A)] $R/I$ is a noetherian right $k$-module;
\item[(3B)] $R/I$ is an artinian right $k$-module;
\item[(4)] $R/I$ is a right $k$-module of finite length;
\item[(5)] $R/I$ is a right $k$-module of finite uniform dimension.
\end{itemize}
The family $\F$ of right ideals satisfying that fixed property is right Oka by
Remark~\ref{creating Oka families}; hence $\Max(\F')$ consists of completely
prime right ideals.
\end{example}

As a refinement of~(4) above, notice that the right $k$-modules
of finite length whose composition factors have certain prescribed isomorphism
types is closed under extensions. The same is true for the right $k$-modules
whose length is a multiple of a fixed integer~$d$. Thus these classes give
rise to two other Oka families of right ideals.

\separate

Right Oka families and completely prime right ideals also arise in connection
with multiplicatively closed subsets of a ring.

\begin{example}
\label{S-torsion example}
Consider a multiplicative subset $S$ of a ring $R$ (i.e., $S$ is a submonoid of the
multiplicative monoid of $R$). A module $M_R$ is said
to be \emph{$S$-torsion} if, for every $m\in M$ there exists $s\in S$ such
that $ms=0$. It is easy to see that the class of $S$-torsion modules is
closed under extensions. Thus the family $\F$ of right ideals $I_R\subseteq R$
such that $R/I$ is $S$-torsion is a right Oka family, and $\Max(\F')$
consists of completely prime right ideals.

Recall that a multiplicative set $S$ in a ring $R$ is called a
\emph{right Ore set} if, for all $a\in R$ and $s\in S$,
$aS\cap sR\neq\varnothing$. (For example, it is easy to see that any
multiplicative set in a commutative ring is right Ore.) One can show that
a multiplicative set $S$ is right Ore iff for every module $M_R$ the set
\[
t_S (M) := \{ m \in M : ms=0 \text{ for some } s \in S \}
\]
of \emph{$S$-torsion} elements of $M$ is a submodule of $M$. This makes it
easy to verify that for such $S$, $R/I$ is $S$-torsion iff $I\cap S\neq\varnothing$.
So for a right Ore set $S\subseteq R$, the family of all right ideals $I$
of $R$ such that $I\cap S\neq\varnothing$ is equal to the family $\F$ above
and thus is a right Oka family. In particular, a right ideal maximal with
respect to being disjoint from $S$ is completely prime. We will be able to
strengthen these statements later---see Example~\ref{Ore set example}.
\end{example}

\separate

One further right Oka family comes from the notion of invertibility of right
ideals. Fix a ring $Q$ with a subring $R\subseteq Q$. For
any submodule $I_R\subseteq Q_R$ we write $I^*:=\{q\in Q:qI\subseteq R\}$,
which is a \emph{left} $R$-submodule of $Q$. We will say that a right
$R$-submodule $I\subseteq Q$ is \emph{right invertible (in $Q$)} if there exist
$x_1,\dots,x_n\in I$ and $q_1,\dots,q_n\in I^*$ such that $\sum x_i q_i =1$.
(This definition is inspired by~\cite[\S II.4]{Stenstroem}.) Notice that if
$I$ is right invertible as above, then $I$ is necessarily finitely generated,
with generating set $x_1,\dots,x_n$. The concept of a right invertible right
ideal certainly generalizes the notion of an invertible ideal in a commutative
ring, and it gives rise to a new right Oka family.

\begin{proposition}
\label{invertible right ideals}
Let $R$ be a subring of a ring $Q$. The family $\F$ of right ideals of $R$
that are right invertible in $Q$ is a right Oka family. The set $\Max(\F')$
consists of completely prime right ideals.
\end{proposition}

\begin{proof}
Let $I_R\subseteq R$ and $a\in R$ be such that $I+aR$ and $a^{-1}I$ are
right invertible. We want to show that $I$ is also right invertible. There
exist $i_1,\dots,i_m\in I$ and $q_k,\ q\in (I+aR)^*$ such that
$\sum_{i=1}^m i_k q_k+aq=1$. Similarly, there exist $x_1,\dots,x_n\in a^{-1}I$
and $p_j\in (a^{-1}I)^*$ such that $\sum x_j p_j=1$. Combining these equations,
we have
\begin{align*}
1 &= \sum i_k q_k + aq \\
&= \sum i_k q_k + a \left( \sum x_j p_j \right) q\\
&= \sum i_k q_k + \sum (a x_j) (p_j q).
\end{align*}
In this equation we have $i_k\in I$, $q_k\in (I+aR)^*\subseteq I^*$, and
$ax_j\in a(a^{-1}I)\subseteq I$. Thus we will be done if we can show that
every $p_jq\in I^*$.

We claim that $qI\subseteq a^{-1}I$. This follows from the fact that, for
any $i\in I$, $q_k i\in R$ so that
\[
aqi = \left( 1 - \sum i_k q_k \right) i = i - \sum i_k (q_k i) \in I.
\]
Thus we find
\[
(p_j q)I = p_j (qI) \subseteq (a^{-1}I)^* (a^{-1}I) \subseteq R.
\]
It follows that $p_jq\in I^*$, completing the proof.
\end{proof}

In the case that $R$ is a right Ore ring, it is known (see~\cite[II.4.3]{Stenstroem})
that the right ideals of $R$ that are right invertible in its classical right ring of
quotients $Q$ are precisely the projective right ideals that intersect the right
Ore set $S$ of regular elements of $R$. 
(Recall that a ring $R$ is right Ore if the multiplicatively closed set of regular
elements in $R$ is right Ore. This is equivalent to the statement that $R$
has a classical right ring of quotients $Q$; see~\cite[\S10B]{Lectures}.)
We can use this to
give a second proof that the family $\F$ of right invertible right ideals of
$R$ is a right Oka family in this case. The alternative characterization of right
invertibility in this setting means that $\F$ is the intersection of the family
$\F_1$ of projective right ideals (which was shown to be a right Oka family as
an application of Example~\ref{proj/inj/flat dim example}) with the family $\F_2$
of right ideals that intersect the right Ore set $S$ (which was shown to be a
right Oka family in Example~\ref{S-torsion example}). Recalling
Remark~\ref{complete lattice}, we conclude that $\F=\F_1\cap\F_2$ is a right Oka
family.

Using this notion of invertibility, we can generalize Theorem~\ref{Cohen's invertible theorem},
due to Cohen, which states that a commutative ring $R$ is a Dedekind domain iff
every nonzero prime ideal of $R$ is invertible.

\begin{proposition}
For a subring $R$ of a ring $Q$, every nonzero right ideal of $R$ is right
invertible in $Q$ iff every nonzero completely prime right ideal of $R$ is right
invertible in $Q$. If $R$ is a right Ore ring with classical right ring of
quotients $Q$, then $R$ is a right hereditary right noetherian domain iff every
nonzero completely prime right ideal of $R$ is right invertible in $Q$.
\end{proposition}

\begin{proof}
First suppose that $R$ is right Ore. According to~\cite[Prop.~II.4.3]{Stenstroem},
a right ideal of $R$ is right invertible in $Q$ iff it is projective and contains a
regular element. Thus the right Ore ring $R$ is a right hereditary right
noetherian domain iff every nonzero right ideal of $R$ is right invertible in
the classical right quotient ring of $R$. So it suffices to prove the first
statement.

Now for any $R$ and $Q$, the family $\F$ of right ideals of $R$ that are right invertible in $Q$ is
a right Oka family by Proposition~\ref{invertible right ideals}. Once we recall
that a right invertible right ideal is finitely generated, the claim follows
from the CPIP Supplement~\ref{CPIP supplement}(2).
\end{proof}

\section{Comonoform right ideals and divisible right Oka families}
\label{comonoform section}

We devote the final section of this paper to study a particularly well-behaved
subset of the completely prime right ideals of a general ring, the \emph{comonoform
right ideals} (Definition~\ref{(co)monoform definition}). Our purpose is to provide a
richer understanding of the completely prime right ideals of a general ring.
There is a special type of Prime Ideal Principle that accompanies this new set of right
ideals, as well as new applications to the one-sided structure of rings.

These special right ideals $I_R\subseteq R$ are defined by imposing a certain condition
on the factor module $R/I$. First we must describe the many equivalent ways to
phrase this module-theoretic condition.
Given an $R$-module $M_{R}$, a submodule $N\subseteq M$ is said to be \emph{dense}
if, for all $x,y\in M$ with $x\neq0$, $x\cdot (y^{-1}N)\neq0$ (recall the definition
of $y^{-1}M$ from~\S\ref{completely prime section}). We write $N\subseteq_{d}M$
to mean that $N$ is a dense submodule of $M$, and we let $E(M)$ denote the injective
hull of $M$.
It is known that $N\subseteq_{d}M$ iff $\Hom_{R}(M/N,E(M))=0$, iff for every
submodule $U$ with $N\subseteq U\subseteq M$ we have $\Hom_R(U/N,M)=0$.
In addition, for any submodules $N\subseteq U\subseteq M$, it turns out that
$N\subseteq_{d}M$ iff $N\subseteq_{d}U$ and $U\subseteq_{d}M$.
(See~\cite[(8.6)~\&~(8.7)]{Lectures} for details.)
Finally, any dense submodule of $M$ is \emph{essential} in $M$, meaning that it
has nonzero intersection with every nonzero submodule of $M$.

\begin{proposition}
\label{monoform characterization}
For a module $M_{R}\neq0$, the following are equivalent:
\begin{itemize}
\item[\normalfont (1)] Every nonzero submodule of $M$ is dense in $M$;
\item[\normalfont (2)] Every nonzero cyclic submodule of $M$ is dense in $M$;
\item[\normalfont (3)] For any $x,y,z\in M$ with $x,z\neq0$, $x\cdot y^{-1}(zR)\neq 0$;
\item[\normalfont (4)] Any nonzero $f\in\Hom_{R}(M,E(M))$ is injective;
\item[\normalfont (5)] For any submodule $C\subseteq M$, any nonzero $f\in\Hom_R(C,M)$
(resp.\ any nonzero $f\in\Hom_R(C,E(M))$) is injective;
\item[\normalfont ($5'$)] $M$ is uniform and for any cyclic submodule $C\subseteq M$, any
nonzero $f\in\Hom_R(C,M)$ is injective;
\item[\normalfont (6)] There is no nonzero $R$-homomorphism from any submodule of any
proper factor of $M$ to $M$ (resp. to $E(M)$).
\end{itemize}
\end{proposition}

\begin{proof}
Clearly (1)$\implies$(2). For (2)$\implies$(1), let $P$ be any nonzero
submodule of $M$. Then, taking some cyclic submodule $0\neq C\subseteq P$,
we have $C\subseteq_d M\implies P\subseteq_d M$.

Now (2)$\iff$(3) is clear from the definition of density. Also,
(1)$\iff$(4) and (1)$\iff$(5) follow from the various reformulations of
density stated above. The equivalence of~(5),~(6), and their parenthetical
formulations is straightforward.

Finally we prove (5)$\iff$($5'$). Assume~(5) holds; to verify~($5'$),
we only need to show that $M$ is uniform. By the equivalence of~(1) and~(5),
we see that every nonzero submodule of $M$ is dense and is therefore
essential. This proves that $M$ is uniform. Now suppose that~($5'$) holds,
and let $0\neq f\in\Hom(C,M)$ where $C$ is any submodule of $M$. Fix some cyclic
submodule $0\neq C_{0}\subseteq C$ such that $C_{0}\nsubseteq\ker f$,
and let $g$ denote the restriction of $f$ to $C_{0}$. By hypothesis,
$0=\ker g=\ker f\cap C_{0}$. Because $M$ is uniform this implies that
$\ker f=0$, proving that~(5) is true.
\end{proof}

An easy example shows that the requirement in ($5'$) that $M$ be uniform
is in fact necessary. If $V_k$ is a vector space over a
division ring $k$ then it is certainly true that every nonzero homomorphism
from a cyclic submodule of $V$ into $V$ is injective. However, if
$\dim_{k}V>1$, then $V$ has nontrivial direct summands and cannot be
uniform.

\begin{definition}
\label{(co)monoform definition}
A nonzero module $M_R$ is said to be \emph{monoform} (following~\cite{GordonRobson})
if it satisfies the equivalent conditions of Proposition~\ref{monoform characterization}. A
right ideal $P_R\subsetneq R$ is \emph{comonoform} if the factor module
$R/P$ is monoform.
\end{definition}

As a basic example, notice that simple modules are monoform and hence
maximal right ideals are comonoform. 
We can easily verify that the comonoform right ideals of a ring form a
subset of the set of completely prime right ideals, as mentioned earlier.

\begin{proposition}
\label{comonoform is completely prime}
If $M_R$ is monoform, then every nonzero endomorphism of $M$ is injective.
In particular, every comonoform right ideal of $R$ is completely prime.
\end{proposition}

\begin{proof}
The first claim follows from Proposition~\ref{monoform characterization}(5)
by taking $C=M$ there. Now the second statement is true by
Proposition~\ref{completely prime characterization}.
\end{proof}

Some clarifying remarks about terminology are appropriate.
Monoform modules have been given several other names in the literature.
They seem to have been first investigated by O.~Goldman in~\cite[\S6]{Goldman}.
Each monoform module is associated to a certain \emph{prime right Gabriel filter}
$\F$ (a term which we will not define here), and Goldman referred to such a module
as a \emph{supporting module for $\F$}. They have also been referred to as
\emph{cocritical modules}, \emph{$\F$-cocritical modules}, and \emph{strongly
uniform modules}. The latter term is justified because, as shown in ($5'$) above,
any monoform module is uniform. Also, comonoform right ideals have been referred
to as \emph{critical right ideals}~\cite{LambekMichler} (which explains the term
``cocritical module'') and \emph{super-prime right ideals}~\cite{Popescu}. We
have chosen to use the term ``monoform'' because we feel that it best describes
the properties of these modules, and we are using the term ``comonoform''
rather than ``critical'' for right ideals in order to avoid confusion with
the modules that are critical in the sense of the Gabriel-Rentschler Krull
dimension.

\separate

Comonoform right ideals enjoy special properties that distinguish them
from the more general completely prime right ideals. For instance, if $P$
is a comonoform right ideal of $R$, then $R/P$ is uniform by Proposition~\ref{monoform characterization}($5'$).
On the other hand, Example~\ref{completely prime not meet-irreducible} showed
that the more general completely prime right ideals do not always have this property.
A second desirable property of comonoform right ideals is given in the
following lemma.
It is easy to verify (from several of the characterizations in Proposition~\ref{monoform characterization}) that a
nonzero submodule of a monoform module is again monoform. Applying
Lemma~\ref{cyclic module lemmas}(A) yields the following result.

\begin{lemma}
For any comonoform right ideal $P_R\subseteq R$ and any element $x\in R\setminus P$,
the right ideal $x^{-1}P$ is also comonoform.
\end{lemma}

It is readily verified that the lemma above does not hold if we replace the
word ``comonoform'' with ``completely prime.'' For instance, consider
again Example~\ref{completely prime not meet-irreducible}. For the completely
prime right ideal $P$ of the ring $R$ described there and the element
$x=E_{12}+E_{13}\in R$, it is readily verified that
\[
x^{-1}P = 
\begin{pmatrix}
k & k & k\\
0 & 0 & 0\\
0 & 0 & 0
\end{pmatrix}
\]
is not a completely prime right ideal (because the module $R/x^{-1}P$ is decomposable).

Without going into details, we mention that the lemma above suggests that comonoform
right ideals $P$ can be naturally grouped into equivalence classes corresponding to the
isomorphism classes of the injective hulls $E(R/P)$. (This is directly related to Goldman's
notion of primes in~\cite{Goldman}, and is also investigated in~\cite{LambekMichler}.)

\separate

Let us consider a few ways one might find comonoform right ideals in
a given ring $R$. First, we have already seen that
\emph{every maximal right ideal in $R$ is comonoform}. Second,
Remark~\ref{spectrum correspondence} shows that if $I\lhd R$ is an ideal
contained in a right ideal $J$, then \emph{$J$ is a comonoform right ideal of $R$
iff $J/I$ is a comonoform right ideal of $R/I$}. Next, let us examine which
(two-sided) ideals of $R$ are comonoform as right ideals. The following result
seems to have been first recorded (without proof) in~\cite[Prop.~4]{Hudry}.

\begin{proposition}
\label{comonoform ideals}
An ideal $P\lhd R$ is comonoform as a right ideal iff $R/P$ is a right
Ore domain.
\end{proposition}

\begin{proof}
First suppose that $P_R$ is comonoform. Then $P$ is completely prime by
Proposition~\ref{comonoform is completely prime}, so that $R/P$ is a domain by
Proposition~\ref{two-sided completely primes}. Also $R/P$ is right uniform
(see Proposition~\ref{monoform characterization}($5'$)). Now because the domain
$R/P$ is right uniform, it is right Ore.

Conversely, suppose that $S:=R/P$ is a right Ore domain, with right division
ring of quotients $Q$. Then because $E(S_R)=Q_R$, it is easy to see that
every nonzero map
\[
f\in\Hom(S_R,E(S_R))=\Hom(S_S,Q_S)
\]
must be injective. Thus $S_R$ is monoform, completing the proof.
\end{proof}

\begin{remark}
This result makes it easy to construct an example of a ring with a
completely prime right ideal that is not comonoform. Let $R$ be any domain
that is not right Ore (such as the free algebra generated by two elements
over a field), and let $P=0\lhd R$. Then $P_R$ is completely prime (recall
Proposition~\ref{two-sided completely primes}), but it cannot be right comonoform
by the above.

Incidentally, because every monoform module is uniform, the ring
$R$ constructed in Example~\ref{completely prime not meet-irreducible} is an
example of an \emph{artinian} (hence noetherian) ring with a completely
prime right ideal that is not comonoform. This is to be contrasted with the
non-Ore domain example, which is necessarily non-noetherian.
\end{remark}

Another consequence of this result is that the comonoform right ideals
directly generalize the concept of a prime ideal in a commutative ring,
just like the completely prime right ideals.

\begin{corollary}
In a commutative ring $R$, an ideal $P\lhd R$ is comonoform iff it is
a prime ideal.
\end{corollary}

\separate

The Completely Prime Ideal Principle~\ref{CPIP} gives us a method for
exploring the existence of completely prime right ideals. We will provide
a similar tool for studying the existence of the more special comonoform
right ideals in Theorem~\ref{divisible Oka families}. The idea is that comonoform
right ideals occur as the right ideals that are maximal in the complement of
right Oka families that satisfy an extra condition, defined below.

\begin{definition}
A family $\F$ of right ideals in a ring $R$ is \emph{divisible} if, for all
$a\in R$,
\[
I \in \F \implies a^{-1}I \in \F.
\]
\end{definition}

The next lemma is required to prove the ``stronger PIP'' for divisible right
Oka families.

\begin{lemma}
\label{ideal quotient lemma}
In a ring $R$, suppose that $I$ and $K$ are right ideals and that $a\in R$.
Then
\[
K \supseteq a^{-1}I \iff K=a^{-1}J \text{ for some right ideal } J_R\supseteq I.
\]
\end{lemma}

\begin{proof}
If $K=a^{-1}J$ for some $J\supseteq I$, then clearly $K = a^{-1}J\supseteq a^{-1}I$.
Conversely, suppose that $K\supseteq a^{-1}I$. Then for $J:=I+aK$, we claim
that $K=a^{-1}J$. Certainly $K\subseteq a^{-1}J$. So suppose that $x\in a^{-1}J$.
Then $ax\in J=I+aK$ implies that $ax=i+ak$ for some $i\in I,\ k\in K$. Because 
$a(x-k)=i$, we see that $x-k\in a^{-1}I$. Hence $x=(x-k)+k\in a^{-1}I + K = K$.
\end{proof}

\begin{theorem}
\label{divisible Oka families}
Let $\F$ be a divisible right Oka family. Then every $P\in\Max(\F')$
is a comonoform right ideal.
\end{theorem}

\begin{proof}
Let $P\in\Max(\F')$. To show that $R/P\neq 0$ is monoform, it is sufficient
by Proposition~\ref{monoform characterization}(1) to show that every nonzero
submodule $I/P\subseteq R/P$ is dense. That is, for any $0\neq x+P\in R/P$ and
any $y+P\in R/P$, we wish to show that $(x+P)\cdot (y+P)^{-1}(I/P)\neq 0$.
It is straightforward
to see that $(y+P)^{-1}(I/P)=y^{-1}I$. Thus it is enough to show, for any
right ideal $I\supsetneq P$ and elements $x\in R\setminus P$ and $y\in R$,
that $x\cdot y^{-1}I \nsubseteq P$.

Assume for contradiction that $x\cdot y^{-1}I \subseteq P$ for such $x$,
$y$, and $I$. Then $y^{-1}I\subseteq x^{-1}P$, and Lemma~\ref{ideal quotient lemma}
shows that $x^{-1}P=y^{-1}J$
for some right ideal $J\supseteq I$. Since $P\in\Max(\F')$, the fact
that $J\supseteq I\supsetneq P$ implies that $J\in\F$. Because $\F$ is
divisible, $x^{-1}P=y^{-1}J\in\F$. Also $x\notin P$ and maximality of
$P$ give $P+xR\in\F$. Since $\F$ is right Oka we conclude that $P\in\F$,
a contradiction.
\end{proof}

As with the more general completely prime right ideals, there is a ``Supplement''
that accompanies this ``stronger PIP.'' 
We omit its proof, which parallels that of Theorem~\ref{CPIP supplement}.

\begin{theorem}
\label{DPIP supplement}
Let $\F$ be a divisible right Oka family in a ring $R$ such that every nonempty
chain of right ideals in $\F'$ (with respect to inclusion) has an upper bound in
$\F'$. Let $\setS$ denote the set of comonoform right ideals of $R$.
\begin{itemize}
\item[\normalfont (1)] Let $\F_{0}$ be a semifilter of right ideals in $R$. If
$\setS\cap\F_{0}\subseteq\F$, then $\F_{0}\subseteq \F$.
\item[\normalfont (2)] For $J_{R}\subseteq R$, if all right ideals in $\setS$
containing $J$ (resp.\ properly containing $J$) belong to $\F$, then all right
ideals containing $J$ (resp.\ properly containing $J$) belong to $\F$.
\item[\normalfont (3)] If $\setS\subseteq\F$, then $\F$ consists of all right
ideals of $R$.
\end{itemize}
\end{theorem}

To apply the last two theorems, we must provide some ways to construct
divisible right Oka families. The first method is extremely straightforward.

\begin{remark}
Let $\E\subseteq\M_R$ be a class of right $R$-modules that is closed under
extensions and closed under passing to submodules.
Then for $\C:=\E\cap\M_R^c$, the right Oka family $\F_{\C}$ is divisible. For if
$I\in\F_{\C}$ and $x\in R$, then $R/x^{-1}I$ is isomorphic to the submodule
$(I+xR)/I$ of $R/I\in\C\subseteq\E$. Then by hypothesis, $R/x^{-1}I\in\E\cap\M_R^c=\C$,
proving that $x^{-1}I\in\F_{\C}$.
\end{remark}

The above method applies immediately to many of the families $\F_{\C}$ which we
have already investigated. To begin with, for a $k$-ring $R$, all of the finiteness
properties listed in Example~\ref{finiteness conditions}  pass to submodules, with
the exception of finite generation~(1A). Thus a right ideal $I$ maximal with respect
to $R/I$ not having one of those properties is comonoform.

We can apply this specifically to rings with the so-called \emph{right
restricted minimum condition}; these are the rings $R$ such that $R/I$ is an
artinian right $R$-module for all right ideals $I\neq 0$. If such a ring $R$
is not right artinian, we see that the zero ideal is in $\Max(\F')$ where $\F$
is the divisible Oka family of right ideals $I\subseteq R$ such that $R/I$ is
an artinian $R$-module. Thus the zero ideal is right comonoform by
Theorem~\ref{divisible Oka families}. Hence $R$ is a right Ore domain by
Proposition~\ref{comonoform ideals}. The fact that such a ring is a right
Ore domain was proved by A.\,J.~Ornstein in~\cite[Thm.~13]{Ornstein} as
a generalization of a theorem of Cohen~\cite[Cor.~2]{Cohen}.

\begin{corollary}[Ornstein]
\label{RRM domain}
If a ring $R$ satisfies the right restricted minimum condition and is not right
artinian, then $R$ is a right Ore domain.
\end{corollary}

In addition, for a multiplicative set $S\subseteq R$ the class of $S$-torsion
modules (see Example~\ref{S-torsion example}) is closed under extensions and
submodules. So a right ideal $I$ maximal with respect to $R/I$ not being
$S$-torsion is comonoform. Notice that this is true {\em whether or not
the set $S$ is right Ore.} (However, if $S$ is not right Ore then we do not
have the characterization that $R/I$ is $S$-torsion iff $I\cap S\neq\varnothing$.)

For another example, fix a multiplicative set $S\subseteq R$, which again
need not be right Ore. A module $M_R$ is said to be \emph{$S$-torsionfree}
if, for any $m\in M$ and $s\in S$, $ms=0$ implies $m=0$. The class of
$S$-torsionfree modules is easily shown to be closed under extensions. Hence
the family $\F$ of right ideals in $R$ such that $R/I$ is $S$-torsionfree
is a right Oka family. Notice that $\F$ can alternatively be described as
\[
\F = \{ I_R \subseteq R : \textnormal{for $r\in R$ and $s\in S$, } rs \in I \implies r \in I \}.
\]
Furthermore, $\F$ is divisible because any submodule of a torsionfree
module is torsionfree. So every right ideal $P\subseteq R$ with $P\in\Max(\F')$ is
comonoform. 

\separate

A second effective method of constructing a divisible right Oka family is by
defining it in terms of certain families of two-sided ideals. This is achieved
in Proposition~\ref{(P_1) family of ideals} below.
Given a right ideal $I$ of $R$, recall that the largest ideal of $R$ contained
in $I$ is called the \emph{core of $I$,} denoted $\core(I)$. It is straightforward
to check that $\core(I)=\ann(R/I)$ for any $I_R\subseteq R$. 

\begin{lemma}
\label{divisible semifilters generated by ideals}
Let $\F$ be a semifilter of right ideals in a ring $R$ that is generated
as a semifilter by two-sided ideals---that is to say, there exists a set
$\G$ of two-sided ideals of $R$ such that
\begin{align*}
\F &= \{I_R\subseteq R:I\supseteq J\text{ for some }J\in\G\} \\
&= \{I_R \subseteq R : \core(I) \in \G\}.
\end{align*}
Then $\F$ is divisible.
\end{lemma}

\begin{proof}
The equality of the two descriptions of $\F$ above follows from the fact
that $\G$ is a semifilter.
Suppose that $I\in\F$, so that there exists $J\in\G$ such that $I\supseteq J$.
Then for any $a\in R$, $aJ\subseteq J\subseteq I$ implies that $J\subseteq a^{-1}I$.
It follows that $a^{-1}I\in\F$, and $\F$ is divisible.
\end{proof}

By analogy with Definition~\ref{semifilter definition}, we define a \emph{semifilter of (two-sided) ideals}
in a ring $R$ to be a family $\G$ of ideals of $R$ such that, for $I,J\lhd R$,
$I\in\G$ and $J\supseteq I$ imply $J\in\G$. As in~\cite{LR} we define the following
property of a family $\G$ of two-sided ideals in $R$:
\begin{itemize}
\item[$(P_1)$:] $\G$ is a semifilter of ideals that is closed under pairwise products
and that contains the ideal $R$ (equivalently, is nonempty).
\end{itemize}
In~\cite[Thm.~2.7]{LR}, it was shown that any $(P_1)$ family of
ideals in a commutative ring is an Oka family. 
The following shows how to define a right Oka family from a $(P_1)$ family of
ideals in a noncommutative ring.

\begin{proposition}
\label{(P_1) family of ideals}
Let $\G$ be a family of ideals in a ring $R$ satisfying $(P_1)$. Then the
semifilter $\F$ of right ideals generated by $\G$ (as in
Lemma~\ref{divisible semifilters generated by ideals}) is a divisible right
Oka family. Thus, every right ideal in $\Max(\F')$ is comonoform.
\end{proposition}

\begin{proof}
Let $\E$ be the class of right $R$-modules $M$ such that $\ann(M)\in\G$.
We claim that $\E$ is closed under extensions in $\M_R$. Indeed, let $L,\ N\in\E$ and
suppose $0\to L\to M\to N\to 0$ is an exact sequence of right $R$-modules.
We want to conclude that $M\in\E$. Because $\ann(L)$ and $\ann(N)$ belong
to $\G$, the fact that $\G$ is $(P_1)$ means that
$\ann(M)\supseteq\ann(N)\cdot\ann(L)$ must also lie in $\G$. Thus $M\in\E$
as desired.

Now any cyclic module $R/I$ has annihilator $\ann(R/I)=\core(I)$. So for
$\C:=\E\cap\M_R^c$ we see that our family is $\F=\F_{\C}$. Hence $\F$ is a
right Oka family.
Lemma~\ref{divisible semifilters generated by ideals} implies that $\F$ is
divisible. The last sentence of the proposition now follows from
Theorem~\ref{divisible Oka families}.
\end{proof}

We will apply the result above to a special example of such a family $\G$
of ideals. For a ring $R$, recall that a subset $S\subseteq R$ is called an
\emph{$m$-system} if $1\in S$ and for any $s,t\in S$ there exists $r\in R$
such that $srt\in S$. It is well-known that an ideal $P\lhd R$ is prime iff
$R\setminus P$ is an $m$-system.

\begin{corollary}
\label{m-system corollary}
\textnormal{(1)} For an $m$-system $S$ in a ring $R$, the family $\F$ of
right ideals $I$ such that $\core (I) \cap S \neq \varnothing$ is a divisible
right Oka family. A right ideal maximal with respect to having its core disjoint
from $S$ is comonoform. 

\textnormal{(2)} For a prime ideal $P$ of a ring $R$, the family of all
right ideals $I_R$ such that $\core(I)\nsubseteq P$ is a divisible right Oka family.
A right ideal $I$ maximal with respect to $\core(I)\subseteq P$ is comonoform.
In particular, if $R$ is a prime ring, a right ideal maximal with respect to
$\core(I)\neq 0$ is comonoform.
\end{corollary}

\begin{proof}
For (1), we can apply Proposition~\ref{(P_1) family of ideals} to the
family $\G$ of ideals having nonempty intersection with the $m$-system $S$, which is
certainly a $(P_1)$ family of ideals. Then~(2) follows from~(1) if we
let $S=R\setminus P$, which is an $m$-system when $P$ is a prime ideal.
\end{proof}

Another application of Proposition~\ref{(P_1) family of ideals} involves
the notion of boundedness.
Recall that a ring $R$ is said to be \emph{right bounded} if every essential
right ideal contains a two-sided ideal that is right essential. (Another way
to say this is that if $I_R\subseteq R$ is essential, then $\core(I)$
is right essential.) Then one can characterize whether certain types of rings
are right bounded in terms of their comonoform right ideals. 
Given a module $M_R$, we write $N\subseteq_{e}M$ to mean that $N$ is an
essential submodule of $M$.

\begin{proposition}
\label{bounded criterion}
Let $R$ be a ring in which the set of ideals $\{J \lhd R : J_R \subseteq_e R_R\}$
is closed under squaring (e.g.\ a semiprime ring or a right nonsingular ring),
and suppose that every ideal of $R$ that is right essential is finitely generated
as a right ideal (this holds, for instance, if $R$ is right noetherian). Then
$R$ is right bounded iff every essential comonoform right ideal of $R$ has right
essential core. 
\end{proposition}

\begin{proof}
Assume that $R$ satisfies the two stated hypotheses. We claim that the ideal
family $\{J\lhd R : J_R\subseteq_e R\}$ is in fact closed under pairwise products.
Indeed, if $I,J\lhd R$ are essential as right ideals, then their product $IJ$
contains the essential right ideal $(I\cap J)^2$ and thus is right essential.
This allows us to apply Proposition~\ref{(P_1) family of ideals} to say that the
family $\F$ of right ideals with right essential core is a divisible right Oka family.
Next, the assumption that every ideal that is right essential is right finitely
generated implies that the union of any nonempty chain of right ideals in $\F'$
lies in $\F'$. Also, the set $\F_0$ of essential right ideals is a semifilter. Thus
the statement of the proposition, excluding the first parenthetical remark, follows
from Theorem~\ref{DPIP supplement}(1).

It remains to verify that a semiprime or right nonsingular ring $R$ satisfies
the first hypothesis. Suppose that $J\lhd R$ is right essential, and let
$I_R$ be a right ideal such that $I\cap J^2=0$. Then
\[
(I \cap J)^2 \subseteq I \cap J^2 =0 \quad \text{and}
\quad (I \cap J) J \subseteq I \cap J^2 = 0
\]
respectively imply that $I \cap J$ squares to zero and has essential right
annihilator. Thus if $R$ is either semiprime or right nonsingular, then
$I\cap J=0$. Because $J$ is right essential, we conclude that $I=0$. Hence
$J^2$ is right essential as desired.
\end{proof}

\begin{corollary}
\label{bounded prime criterion}
A prime right noetherian ring $R$ is right bounded iff every essential
comonoform right ideal of $R$ has nonzero core.
\end{corollary}

\begin{proof}
It is a well-known (and easy to verify) fact that every nonzero ideal of a
prime ring is right essential. Thus a right ideal of $R$ has right essential
core iff its core is nonzero. Because $R$ is prime and right noetherian, we
can directly apply Proposition~\ref{bounded criterion}.
\end{proof}

In fact, the last result can be directly deduced from Corollary~\ref{m-system corollary}(2).
We chose to include Proposition~\ref{bounded criterion} because it seems to
apply rather broadly.

\separate

Next we will show that certain well-studied families of right ideals are
actually examples of divisible right Oka families, providing a third method of
constructing the latter.
The concept of a \emph{Gabriel filter} of right ideals arises naturally in the study
of torsion theories and the related subject of localization in noncommutative rings.
The definition of these families is recalled below.

\begin{definition}
\label{Gabriel filter definition}
A \emph{right Gabriel filter} (or \emph{right Gabriel topology}) in a ring
$R$ is a nonempty family $\F$ of right ideals of $R$ satisfying
the following four axioms (where $I_R,J_R\subseteq R$):
\begin{itemize}
\item[\normalfont (1)] If $I\in \F$ and $J\supseteq I$ then $J\in \F$;
\item[\normalfont (2)] If $I,J\in \F$ then $I\cap J\in \F$;
\item[\normalfont (3)] If $I\in \F$ and $x\in R$ then $x^{-1}I\in \F$;
\item[\normalfont (4)] If $I\in \F$ and $J_{R}\subseteq R$ is such that
$x^{-1}J\in \F$ for all $x\in I$, then $J\in \F$.
\end{itemize}
\end{definition}

Notice that axiom~(3) above simply states that \emph{a right Gabriel filter is divisible}.
For the reader's convenience, we outline some basic facts regarding right Gabriel filters
and torsion theories that will be used here. Refer to~\cite[VI.1-5]{Stenstroem}
for further details.

Given any right Gabriel filter $\F$ and any module $M_R$, we define a subset
of $M$:
\[
t_{\F}(M) := \{m\in M : \ann (m)\in\F\}.
\]
Axioms (1), (2), and (3) of Definition~\ref{Gabriel filter definition}
guarantee that this is a submodule of $M$, and it is called the
\emph{$\F$-torsion submodule} of $M$. A module $M_R$ is defined to be
\emph{$\F$-torsion} if $t_{\F}(M)=M$ or \emph{$\F$-torsionfree} if
$t_{\F}(M)=0$.
One can easily verify that {\em for a right Gabriel filter $\F$, a right
ideal $I\subseteq R$ lies in $\F$ iff $R/I$ is $\F$-torsion}. 

For any Gabriel filter $\F$, it turns out that the class 
\[
\T_{\F}:=\{ M_R : \text{$M$ is $\F$-torsion, i.e.\ } M=t_{\F}(M)\} 
\]
of all $\F$-torsion right $R$-modules satisfies the axioms of a
\emph{hereditary torsion class}. While we shall not define this term
here, it is equivalent to saying that the class $\T_{\F}$ is
\emph{closed under factor modules, direct sums of arbitrary families, and
extensions (in $\M_R$).} (Thus the reader may simply take this to be the
definition of a hereditary torsion class.)

With the information provided above we will prove that right Gabriel filters
are examples of divisible right Oka families.

\begin{proposition}
\label{GPIP}
Over a ring $R$, any right Gabriel filter $\F$ is a divisible right Oka family.
Any right ideal $P\in\Max(\F')$ is comonoform.
\end{proposition}

\begin{proof}
Any right Gabriel filter is tautologically a divisible family of right ideals.
The torsion class $\T_{\F}$ is closed under extensions in $\M_R$, so the class
$\C:=\T_{\F}\cap\M_R^c$ of cyclic $\F$-torsion modules is closed under extensions.
A right ideal $I\subseteq R$ lies in $\F$ iff $R/I\in\T_{\F}$ (as mentioned above),
iff $R/I\in\C$ (since $R/I\in\M_R^c$), iff $I\in\F_{\C}$.
It follows from Theorem~\ref{extension correspondence} that $\F=\F_{\C}$ is a right
Oka family.
The last sentence is true by Theorem~\ref{divisible Oka families}.
\end{proof}

We pause for a moment to give a sort of ``converse'' to this result, in the
spirit of Proposition~\ref{CPIP converse}.
Given any injective module $E_{R}$, the class $\{M_{R}:\Hom(M,E)=0\}$ is a
hereditary torsion class. This is called the \emph{torsion class cogenerated by $E$}.
We will also say that the corresponding right Gabriel filter is the \emph{right Gabriel
filter cogenerated by $E$}. As stated in~\cite[VI.5.6]{Stenstroem}, this is
the largest right Gabriel filter with respect to which $E$ is torsionfree. 
Let $I$ be a right ideal in $R$. In the following, we let $\F_I$ denote the
right Gabriel filter cogenerated by $E(R/I)$; that is, $\F_I$ is the set
of all right ideals $J\subseteq R$ such that $\Hom_R(R/J,E(R/I))=0$.  
We are now ready for the promised result. 

\begin{proposition}
\label{GPIP converse}
For any right ideal $P\subsetneq R$, the following are equivalent:
\begin{itemize}
\item[\normalfont (1)] $P\in\Max (\F ')$ for some right Gabriel filter $\F$;
\item[\normalfont (2)] $P\in\Max (\F_P^{\prime})$;
\item[\normalfont (3)] $P$ is a comonoform right ideal.
\end{itemize}
\end{proposition}

\begin{proof}
(2)$\implies$(1) is clear, and (1)$\implies$(3) follows from Theorem~\ref{GPIP}.
For (3)$\implies$(2), assume that $R/P$ is monoform. Proposition~\ref{monoform characterization}
implies that for every right ideal $I\supsetneq P$ we have $\Hom_R(R/I,E(R/P))=0$.
Then every such right ideal $I$ tautologically lies in $\F_P$, proving that $P\in\Max(\F_P^{\prime})$.
\end{proof}

We mention in passing that this result is similar to~\cite[Thm.~2.9]{GordonRobson},
though it is not stated in quite the same way.
This proposition actually provides a second, though perhaps less satisfying, proof
that any comonoform right ideal is completely prime. Given a comonoform right ideal
$P\subseteq R$, Proposition~\ref{GPIP converse} provides a right Gabriel filter $\F$
with $P\in\Max(\F')$. Then because $\F$ is a right Oka family (by Theorem~\ref{GPIP}),
the CPIP implies that $P$ is a completely prime right ideal.

\separate

As a first application of Theorem~\ref{GPIP} we explore the maximal point annihilators
of an \emph{injective} module, recovering a result of Lambek and Michler
in~\cite[Prop.~2.7]{LambekMichler}. This should be compared with
Proposition~\ref{point annihilators}.

\begin{proposition}[Lambek and Michler]
For any injective module $E_R$, a maximal point annihilator of $E$ is comonoform.
\end{proposition}

\begin{proof}
Let $\F=\{I_{R}\subseteq R : \Hom_R (R/I,E)=0\}$ be the right Gabriel
filter cogenerated by $E$. Then the set of maximal point annihilators of $E$
is clearly equal to $\Max(\F')$. By Theorem~\ref{GPIP}, any $P\in\Max(\F')$
is comonoform.
\end{proof}

\begin{example}
As shown in~\cite[VI.6]{Stenstroem}, the set $\F$ of all dense right ideals
in any ring $R$ is a right Gabriel filter. (In fact, it is the right Gabriel
filter cogenerated by the injective module $E(R_R)$.) Therefore $\F$ is a right
Oka family, and a right ideal maximal with respect to not being dense in
$R$ is comonoform.
Furthermore, in a right nonsingular ring, this family $\F$ coincides with
the set of all essential right ideals (see~\cite[(8.7)]{Lectures}
or~\cite[VI.6.8]{Stenstroem}). Thus in a right nonsingular ring, the family
$\F$ of essential right ideals is a right Gabriel filter, and a right ideal
maximal with respect to not being essential is comonoform.
\end{example}

\begin{example}
\label{Ore set example}
Let $S$ be a right Ore set in a ring $R$, and let $\F$ denote the family of
all right ideals $I_R\subseteq R$ such that $I\cap S\neq\varnothing$. It is
shown in the proof of~\cite[Prop.~VI.6.1]{Stenstroem} that $\F$ is a right
Gabriel filter. It follows from Theorem~\ref{GPIP} that a right ideal maximal
with respect to being disjoint from $S$ is comonoform.
\end{example}

We offer an application of the example above. Let us say that a multiplicative
set $S$ in a ring $R$ is \emph{right saturated} if $ab\in S$ implies
$a\in S$ for all $a,b\in R$.

\begin{corollary}
For every right saturated right Ore set $S\subseteq R$, there exists a set
$\{P_i \}$ of comonoform right ideals such that $R\setminus S=\bigcup P_i$.
\end{corollary}

\begin{proof}
Indeed, for all $x\in R\setminus S$, we must have $xR\subseteq
R\setminus S$ because $S$ is right saturated. By a Zorn's Lemma argument,
there is a right ideal $P_{x}$ containing $x$ maximal with respect to being
disjoint from $S$. Example~\ref{Ore set example} implies that $P_{x}$ is
comonoform. Choosing such $P_{x}$ for all $x\in R\setminus S$, we have
$R\setminus S=\bigcup P_x$.
\end{proof}

Next we apply Example~\ref{Ore set example} to show that a ``nice enough''
prime (two-sided) ideal must be ``close to'' some comonoform right ideal.

\begin{corollary}
Let $P_0\in\Spec(R)$ be such that $R/P_0$ is right Goldie. Then there exists a
comonoform right ideal $P_R\supseteq P_0$ such that $\core(P)=P_0$. In particular,
if $R$ is right noetherian then every prime ideal occurs as the core of some
comonoform right ideal.
\end{corollary}

\begin{proof}
Remark~\ref{spectrum correspondence} shows that, for any right ideal
$L_R\subseteq R$ and any two-sided ideal $I\subseteq L$, $L$ is comonoform in
$R$ iff $L/I$ is comonoform in $R/I$. Then passing to the factor ring $R/P_0$,
it clearly suffices to show that {\em in a prime right Goldie ring $R$ there
exists a comonoform right ideal $P$ of $R$ with zero core.} Indeed, let
$S\subseteq R$ be the set of regular elements, and let $P_R\subseteq R$ be
maximal with respect to $P\cap S=\varnothing$. Because $R$ is prime right
Goldie it is a right Ore ring by Goldie's Theorem, making $S$ a right Ore set.
Then $P$ is comonoform by Example~\ref{Ore set example}. We claim that $\core(P)=0$.
Indeed, suppose that $I\neq 0$ is a nonzero ideal of $R$. Then since $R$ is
prime, $I$ is essential as a right ideal in $R$. It follows from the theory
of semiprime right Goldie rings  that $I\cap S\neq \varnothing$ (see, for
instance,~\cite[(11.13)]{Lectures}). This means that we cannot have
$I\subseteq P$, verifying that $\core(P)=0$.
\end{proof}

Notice that the above condition on $P_0$ is satisfied if $R/P_0$ is right noetherian.
Conversely, it is not true that the core of every comonoform right ideal is
prime, even in an artinian ring. For example, let $R$ be the ring of
$n\times n$ upper-triangular matrices over a division ring $k$ for $n\geq 2$,
and let $P\subsetneq R$ be the right ideal consisting of matrices in $R$
whose first row is zero. Then one can verify that $R/P$ is monoform (for
example, using a composition series argument), so that $P$ is comonoform.
But the ideal $\core(P)=0$ is not (semi)prime.

We also provide a slight variation of Corollary~\ref{domain characterization},
which tested whether or not a ring $R$ is a domain. The version below applies
when $R$ is a right Ore ring.

\begin{proposition}
A right Ore ring $R$ is a domain iff every nonzero comonoform right ideal of
$R$ contains a regular element.
\end{proposition}

\begin{proof}
(``If'' direction) Let $S\subseteq R$ be the set of regular elements of $R$.
Then $S$ is a right Ore set, so the family $\F:=\{I_R\subseteq R : I\cap S\neq\varnothing\}$
is a right Gabriel filter (in particular, a divisible right Oka family) by
Example~\ref{Ore set example}. Clearly the union of a chain of right ideals in
$\F'$ also lies in $\F'$. By Theorem~\ref{CPIP supplement}, if every
nonzero comonoform right ideal of $R$ contains a regular element, then so does
every nonzero right ideal. It follows easily that $R$ is a domain.
\end{proof}

\separate

As a closing observation, we note that there is a second way (aside from
Theorem~\ref{GPIP}) that right Gabriel filters give rise to comonoform
right ideals.
Given a right Gabriel filter $\G$ in a ring $R$, the class of $\G$-torsionfree
modules is closed under extensions and submodules (just as the class of
$\G$-torsion modules was).
Thus a right ideal $I$ of $R$ maximal with respect to the property that $R/I$ 
is not $\G$-torsionfree must be comonoform by Theorem~\ref{divisible Oka families}.
A similar statement was shown to be true for the $S$-torsionfree property, where
$S$ is a multiplicative set. However, there is a logical relation between these
facts only in the case that $S$ is right Ore, when the family $\G$ of right ideals
intersecting $S$ is a right Gabriel filter.

\section*{Acknowledgments}

I am truly grateful to Professor T.\,Y.~Lam for his patience, advice, and encouragement
during the writing of this paper, as well as his help formulating
Proposition~\ref{no extremely primes}.
I also thank Professor G.~Bergman for providing me with many useful comments after a
very careful reading of a draft of this paper, and particularly for finding the first proof of
Proposition~\ref{torsionfree characterization} (and Corollary~\ref{domain characterization}).
Finally, I thank the referee for an insightful review that led to a substantially better
organization of this paper.

\bibliographystyle{amsplain}
\bibliography{CPIPv4}
\end{document}